\documentclass[a4paper,10pt]{article}

\usepackage{amsmath,amsthm}
\usepackage{amssymb,latexsym}
\usepackage{enumerate}
\newtheorem{thm}{Theorem}[section]
\newtheorem{cor}[thm]{Corollary}
\newtheorem{lem}[thm]{Lemma}

\newtheorem{prop}[thm]{Proposition}

\theoremstyle{definition}

\newtheorem{rem}[thm]{Remark}
\newtheorem{exa}[thm]{Example}

\usepackage[mathcal]{euscript}
\usepackage{diagrams}
\usepackage{epsfig}

\newcommand{\im}{\operatorname{im}\nolimits}
\newcommand{\ind}[1]{\operatorname{ind}_{#1}\nolimits}
\newcommand{\proj}[1]{\operatorname{proj}_{#1}\nolimits}
\newcommand{\Projeins}{\ensuremath{\operatorname{Proj}^{\,\operatorname{1}}}}
\newcommand{\Bigcap}[1]{\ensuremath{\mathop{\cap}_{#1}}}
\newcommand{\Bigcup}[1]{\ensuremath{\mathop{\cup}_{#1}}}
\newcommand{\Bigsum}[2]{\ensuremath{\mathop{\textstyle\sum}_{#1}^{#2}}}
\newcommand{\Prod}[2]{\ensuremath{\underset{\scriptscriptstyle #1}{\overset{\scriptscriptstyle#2}{{\textstyle\prod\displaystyle}}}\,}}

\begin{document}

\title{Projective limits of weighted LB-spaces\\of holomorphic functions}

\author{Sven-Ake Wegner}

\date{May 21, 2011}

\maketitle

\renewcommand{\thefootnote}{}

\renewcommand{\thefootnote}{\arabic{footnote}}
\setcounter{footnote}{0}

\vspace{-40pt}
\begin{center}
\textit{Dedicated to the memory of Klaus D.~Bierstedt.}\vspace{15pt}
\end{center}

\begin{abstract}\noindent{}Countable projective limits of countable inductive limits, called PLB-spaces, of weighted Banach spaces of continuous functions have recently been investigated by Agethen, Bierstedt and Bonet. We extend their investigation to the case of holomorphic functions regarding the same type of questions, i.e.~we analyze locally convex properties in terms of the defining double sequence of weights and study the interchangeability of projective and inductive limit. 
\end{abstract}


\section{\hspace{-10pt}Introduction}\label{Introduction}

In this article we investigate the structure of spaces of holomorphic functions defined on an open subset of $\mathbb{C}^d$ that can be written as a countable intersection of countable unions of weighted Banach spaces of holomorphic functions where the latter are defined by weighted sup-norms. The spaces we are interested in are examples of PLB-spaces, i.e.~countable projective limits of countable inductive limits of Banach spaces. Spaces of this type arise naturally in analysis, for instance the space of distributions, the space of real analytic functions and several spaces of ultradifferentiable functions and ultradistributions are of this type. In particular, some of the mixed spaces of ultradistributions (studied recently by Schmets, Valdivia \cite{SchVal2008, SchVal2008a, SchVal2009}) appear to be weighted PLB-spaces of holomorphic functions (see \cite[Section 15]{PHD} for details). In fact, all the forementioned spaces are even PLS-spaces that is the linking maps in the inductive spectra of Banach 
spaces are compact and some of them even appear to be PLN-spaces (i.e.~the linking maps are nuclear). During the last years the theory of PLS-spaces has played an important role in the application of abstract functional analytic methods to several classical problems in analysis. We refer to the survey article \cite{Domanski2004} of Doma\'{n}ski for applications, examples and further references.
\smallskip
\\Many of the applications reviewed by Doma\'{n}ski \cite{Domanski2004} are based on the theory of the so-called first derived functor of the projective limit functor. This method has its origin in the application of homological algebra to functional analysis. The research on this subject was started by Palamodov \cite{Palamodov1971, Palamodov1968} in the late sixties and carried on since the mid eighties by Vogt \cite{VogtLectures} and many others. We refer to the book of Wengenroth \cite{Wengenroth}, who laid down a systematic study of homological tools in functional analysis and in particular presents many ready-for-use results concerning concrete analytic problems. In particular, \cite[Chapter 5]{Wengenroth} illustrates that for the splitting theory of Fr\'{e}chet or more general locally convex spaces, the consideration of PLB-spaces which are not PLS-spaces is indispensable.
\smallskip
\\A major application of the theory of the derived projective limit functor $\Projeins$ is the connection between its vanishing on a countable projective spectrum of LB-spaces and locally convex properties of the projective limit of the spectrum (e.g.~being ultrabornological or barrelled). This connection was firstly noticed by Vogt \cite{VogtLectures, Vogt1989}, see \cite[3.3.4 and 3.3.6]{Wengenroth}, who also gave characterizations of the vanishing of $\Projeins$ and the forementioned properties in the case of sequence spaces, cf.~\cite[Section 4]{Vogt1989}. A natural extension of Vogt's work is to study the case of continuous functions, which was the subject of the thesis of Agethen \cite{Agethen2004}. Recently, an extended and improved version of her results was published by Agethen, Bierstedt, Bonet \cite{ABB2009}. In addition to the study of the projective limit functor, Agethen, Bierstedt, Bonet studied the interchangeability of projective and inductive limit, i.e.~the question when the PLB-spaces are 
equal to the weighted LF-spaces of continuous functions studied for the first time by Bierstedt, Bonet \cite{BB1994}. In view of the results of \cite{ABB2009}, it is a natural objective to extend the investigation on weighted PLB-spaces of holomorphic functions, having in mind the same type of questions.
\smallskip
\\As in the case of sequence spaces and continuous functions the starting point in the definition of weighted PLB-spaces of holomorphic functions is a double sequence of strictly positive and continuous functions (weights); in Section \ref{Preliminaries} we precise the latter and establish further terminology and basic properties of the weighted PLB-spaces of holomorphic functions $AH(G)$ and $(AH)_0(G)$ under O- and o-growth conditions. According to the above, our first aim is then the characterization of locally convex properties of the spaces in terms of the defining sequence. This task is splitted in two parts: In Section \ref{Necessary} we study necessary conditions for barrelledness of $AH(G)$ and $(AH)_0(G)$. This is possible within a setting of rather mild assumptions, whose definition is motivated by the article \cite{BBG} of Bierstedt, Bonet, Galbis. In Section \ref{Sufficient} we turn to sufficient conditions for ultrabornologicity and barrelledness, where we have to decompose holomorphic 
functions in order to utilize the homological methods mentioned earlier and also to apply a criterion explained in \cite{BonetWegner2010} by Bonet, Wegner which we need for our investigation of the o-growth case. Since there is up to now no method available, which allows a decomposition of holomorphic functions in broad generality (in contrast to the case of continuous functions, cf.~\cite[3.5]{ABB2009}), we restrict ourselves to the unit disc and introduce a class of weights for which decomposition is possible. The definition of this class traces back to work of Bierstedt, Bonet \cite{BB2003} and strongly relies on results of Lusky \cite{Lusky1992, Lusky1995}. In Section \ref{Inter}, we study the interchangeability of projective and inductive limit, which is of course closely connected with weighted LF-spaces of holomorphic functions, defined and studied recently by Bierstedt, Bonet \cite{BB2003}. Unfortunately, in all the results of Sections \ref{Necessary} -- \ref{Inter} the necessary conditions are 
slightly but strictly weaker than the sufficient conditions. Therefore, in section 6 we consider the case in which all the weights are essential in the sense of Taskinen and we also assume condition $(\Sigma)$ of Bierstedt and Bonet [6]. Under these further assumptions we provide full characterizations of ultrabornologicity and the interchangeability of projective and inductive limit. The examples we discuss in the final Section 7 illustrate that the new assumptions of section 6 are rather natural.
\smallskip
\\We refer the reader to \cite{BMS1982} for weighted spaces of holomorphic functions and to \cite{KoetheI, KoetheII, MeiseVogtEnglisch} for the general theory of locally convex spaces.


\section{\hspace{-10pt}Notation and preliminary results}\label{Preliminaries}

Let $G$ be an open subset of $\mathbb{C}^d$ and $d\geqslant1$. By $H(G)$ we denote the space of all holomorphic functions on $G$, which we endow with the topology co of uniform convergence on the compact subsets. A \textit{weight} $a$ on $G$ is a strictly positive and continuous function on $G$. For a weight $a$ we define
\begin{align*}
Ha(G)&:=\{\,f\in H(G)\: ; \: \|f\|_a:=\sup_{z\in G}a(z)|f(z)|<\infty\,\},\\
Ha_0(G)&:=\{\,f\in H(G)\: ; \: a|f| \text{ vanishes at } \infty \text{ on } G\,\}.
\end{align*}
Recall that a function $g\colon G\rightarrow \mathbb{R}$ is said to vanish at infinity on $G$ if for each $\varepsilon>0$ there is a compact set $K$ in $G$ such that $|g(z)|<\varepsilon$ holds for all $z\in G\backslash K$. The space $Ha(G)$ is a Banach space for the norm $\|\cdot\|_a$ and $Ha_0(G)$ is a closed subspace of $Ha(G)$. In the first case we speak of \textit{O-growth conditions} and in the second of \textit{o-growth conditions}.
\medskip
\\In order to define the projective spectra we are interested in, we consider a double sequence $\mathcal{A}=((a_{N,n})_{n\in\mathbb{N}})_{N\in\mathbb{N}}$ of weights on $G$ which is decreasing in $n$ and increasing in $N$, i.e.~$a_{N,n+1}\leqslant a_{N,n}\leqslant a_{N+1,n}$ holds for each $n$ and $N$. This condition will be assumed on the double sequence $\mathcal{A}$ in the rest of this article. We define the norms $\|\cdot\|_{N,n}:=\|\cdot\|_{a_{N,n}}$ and hence we have $\|\cdot\|_{N,n+1}\leqslant \|\cdot\|_{N,n}\leqslant \|\cdot\|_{N+1,n}$ for each $n$ and $N$. Accordingly, $Ha_{N,n}(G)\subseteq Ha_{N,n+1}(G)$ and $H(a_{N,n})_0(G)\subseteq H(a_{N,n+1})_0(G)$ holds with continuous inclusions for all $N$ and $n$ and we can define for each $N$ the weighted inductive limits
$$
\mathcal{A}_NH(G):=\ind{n}Ha_{N,n}(G) \; \text{ and } \; (\mathcal{A}_N)_0H(G):=\ind{n}H(a_{N,n})_0(G).
$$
We denote by $B_{N,n}$ the closed unit ball of the Banach space $Ha_{N,n}(G)$, i.e.
$$
B_{N,n}:=\{\,f\in H(G)\:;\:\|f\|_{N,n}\leqslant 1\,\}.
$$
By Bierstedt, Meise, Summers \cite[end of the remark after 1.13]{BMS1982} (cf.~also Bierstedt, Meise \cite[3.5.(2)]{BM1985}) we know that $\mathcal{A}_NH(G)$ is a complete, hence regular LB-space. We will assume without loss of generality, by multiplying by adequate scalars, that every bounded subset $B$ in $\mathcal{A}_NH(G)$ is contained in $B_{N,n}$ for some $n$.
\smallskip
\\The weighted inductive limits $(\mathcal{A}_N)_0H(G)$ need not to be regular. The closed unit ball of the Banach space $H(a_{N,n})_0(G)$ is denoted by
$$
B_{N,n}^{\circ}:=\{\,f\in H(a_{N,n})_0(G)\:;\:\|f\|_{N,n}\leqslant 1\,\}.
$$
For each $N$ we have $\mathcal{A}_{N+1}H(G)\!\subseteq\!\mathcal{A}_NH(G)$ and $(\mathcal{A}_{N+1})_0H(G)\!\subseteq\!(\mathcal{A}_N)_0H(G)$ with continuous inclusions. Then $\mathcal{A}H:=(\mathcal{A}_NH(G),\subseteq_{N+1}^N)_{N\in\mathbb{N}}$ and $\mathcal{A}_0H:=((\mathcal{A}_N)_0H(G),\subseteq_{N+1}^N)_{N\in\mathbb{N}}$ are projective spectra of LB-spaces with inclusions as linking maps. We can now form the following projective limits, called \textit{weighted PLB-spaces of holomorphic functions}
$$
AH(G):=\proj{N}\mathcal{A}_NH(G) \; \text{ and } \; (AH)_0(G):=\proj{N}(\mathcal{A}_N)_0H(G),
$$
which are the object of our study in this work. By definition, $AH(G)\subseteq (AH)_0(G)$ holds with continuous inclusion.
\smallskip
\\We refer the reader to the book of Wengenroth \cite{Wengenroth} for a detailed exposition of the theory of projective spectra of locally convex spaces $\mathcal{X}=(X_N)_{N\in\mathbb{N}}$, their projective limits $\proj{N}X_N$, the derived functor $\Projeins$ and for conditions to ensure that the derived functor on a projective spectrum vanishes, i.e.~that we have $\Projeins\mathcal{X}=0$, including important results of Palamodov \cite{Palamodov1968, Palamodov1971}, Retakh \cite{Retakh1970}, Braun, Vogt \cite{BraunVogt1997}, Vogt \cite{VogtLectures, Vogt1989}, Frerick, Wengenroth \cite{FrerickWengenroth1996} and many others. At this point we only mention that, if $\mathcal{X}=(X_N)_{N\in\mathbb{N}}$ is a projective spectrum of locally convex spaces with inclusions as linking maps and limit $X=\proj{N}X_N$, the so-called fundamental resolution
$$
0\rightarrow X\rightarrow \Prod{N=1}{\infty}X_N\mathop{\rightarrow}^{\sigma}\Prod{N=1}{\infty}X_N,
$$
where $\sigma((x_N)_{N\in\mathbb{N}}):=(x_{N+1}-x_N)_{N\in\mathbb{N}}$, is exact but $\sigma$ is not necessarily surjective, what directs to the definition
$$
\Projeins\mathcal{X}:=\big(\Prod{N=1}{\infty}X_N\big)\big/\im\sigma.
$$
For more details see Wengenroth \cite[Chapter 3]{Wengenroth}.
\smallskip
\\An important tool to handle weighted spaces of holomorphic functions is the technique of associated weights or growth conditions mentioned by Anderson, Duncan \cite{AndersonDuncan1990}, studied for the first time in a systematic way by Bierstedt, Bonet, Taskinen \cite{BBT} and used in many articles dealing with weighted spaces of holomorphic functions. For a given weight $a$ we call $w:=1/a$ the corresponding \textit{growth condition} and define (cf.~\cite[1.1]{BBT}) the \textit{associated growth condition}
$$
\tilde{w}=({\textstyle\frac{1}{a}})^{\sim}\colon G\rightarrow \mathbb{R},\;\; z\,\mapsto\!\sup_{\stackrel{g\in H(G),}{\scriptscriptstyle|g|\leqslant w}}\!\!|g(z)|\,=\sup_{\scriptscriptstyle g\in B_a}|g(z)|.
$$
In \cite[previous to 1.12]{BBT}, Bierstedt, Bonet, Taskinen put $\tilde{a}:=1/\tilde{w}$ and called $\tilde{a}$ the \textit{weight associated with $a$}. However, in most cases we will stick to the first notation. Bierstedt, Bonet, Taskinen (cf.~\cite[4.B after 1.12]{BBT}) introduced as well an associated weight for the case of o-growth conditions by setting $\tilde{w}_0(z)=(1/a)^{\sim}(z):=\sup_{f\in B_a^\circ}|f(z)|$, but in a rather general setting (see Section \ref{Necessary}) both notions coincide.
\medskip
\\In \cite{Vogt1992} Vogt introduced the conditions (Q) and (wQ). In the case of weighted PLB-spaces, these conditions can be reformulated in terms of the weights as follows. We say that the sequence $\mathcal{A}$ satisfies condition (Q) if 
$$
\textstyle\forall \: N \; \exists \: M \geqslant N,\, n \; \forall \: K \geqslant M,\, m,\, \varepsilon > 0 \; \exists \: k, \, S>0 : \frac{1}{a_{M,m}} \leqslant \max\big(\frac{\varepsilon}{a_{N,n}},\frac{S}{a_{K,k}}\big),
$$
we say that it satisfies (wQ) if
$$
\textstyle\forall \: N \; \exists \: M \geqslant N,\, n \; \forall \: K \geqslant M,\, m \; \exists \: k, \, S>0 : \frac{1}{a_{M,m}} \leqslant \max\big(\frac{S}{a_{N,n}},\frac{S}{a_{K,k}}\big).
$$
It is clear that condition (Q) implies condition (wQ). Bierstedt, Bonet gave in \cite{BB1994} an example of a sequence which satisfies (wQ) but not (Q). We define the following conditions by the use of associated weights, where the quantifiers are always those of (wQ) or (Q) resp.~and the estimates are the following:
\begin{align*}
\textstyle\text{(Q)}^{\sim}_{\text{in}}  &\;\;\colon\;\; \textstyle \big(\frac{1}{a_{M,m}}\big)^{\sim} \leqslant \max\big(\big(\frac{\varepsilon}{a_{N,n}}\big)^{\sim},\big(\frac{S}{a_{K,k}}\big)^{\sim}\big)\\
\textstyle\text{(Q)}^{\sim}_{\text{out}} &\;\;\colon\;\; \textstyle \big(\frac{1}{a_{M,m}}\big)^{\sim} \leqslant \big(\max\big(\frac{\varepsilon}{a_{N,n}},\frac{S}{a_{K,k}}\big)\big)^{\sim}\\
\textstyle\text{(wQ)}^{\sim}_{\text{in}} &\;\;\colon\;\; \textstyle \big(\frac{1}{a_{M,m}}\big)^{\sim} \leqslant S\,\max\big(\big(\frac{1}{a_{N,n}}\big)^{\sim},\big(\frac{1}{a_{K,k}}\big)^{\sim}\big)\\
\textstyle\text{(wQ)}^{\sim}_{\text{out}}&\;\;\colon\;\; \textstyle \big(\frac{1}{a_{M,m}}\big)^{\sim} \leqslant S\big(\max\big(\frac{1}{a_{N,n}},\frac{1}{a_{K,k}}\big)\big)^{\sim}
\end{align*}
It follows from \cite[1.2.(vii)]{BBT} that condition $\text{(Q)}^{\sim}_{\text{in}}$ implies $\text{(Q)}^{\sim}_{\text{out}}$ and $\text{(wQ)}^{\sim}_{\text{in}}$ implies $\text{(wQ)}^{\sim}_{\text{out}}$ in general. Moreover, condition (wQ) implies $\text{(wQ)}^{\sim}_{\text{out}}$ and condition (Q) implies $\text{(Q)}^{\sim}_{\text{out}}$.
\smallskip
\\In \cite[1.1]{Vogt1983} Vogt introduced a variant of the following condition to characterize Fr\'{e}chet spaces between which all continuous linear mappings are bounded. According to Vogt, but reformulated for our setting, we say that a sequence $\mathcal{A}$ as above satisfies condition (B) if
$$
\forall \: (n(N))_{N\in\mathbb{N}}\subseteq \mathbb{N} \; \exists \: m \; \forall \: M \; \exists \: L,\, c>0 \colon a_{M,m} \, \leqslant \, c \max_{\scriptscriptstyle N=1,\dots,L} a_{N,n(N)}.
$$
Condition $\text{(B)}^{\sim}$ is defined by the same quantifiers and the estimate replaced by
$$
\tilde{a}_{M,m} \, \leqslant \, c (\max_{\scriptscriptstyle N=1,\dots,L} a_{N,n(N)})^{\sim}.
$$
Again, \cite[1.2.(vii)]{BBT} provides that (B) implies $\text{(B)}^{\sim}$.


\section{\hspace{-10pt}Necessary conditions for barrelledness:\newline{}Results for balanced domains}
\label{Necessary}
In what follows we establish necessary conditions for barrelledness of the spaces $AH(G)$ and $(AH)_0(G)$. This is possible under rather mild assumptions. In \cite{BBG} Bierstedt, Bonet, Galbis studied the following setting: $G$ is balanced, all considered weights are radial (i.e.~for each weight $a$ they assume $a(z)=a(\lambda z)$ for every $\lambda\in\mathbb{C}$ with $|\lambda|=1$), the Banach space topologies are stronger than co and the polynomials are contained in all the considered spaces. They remark that for bounded $G$ the latter is equivalent to requiring that each weight $a_{N,n}$ extends continuously to $\overline{G}$ with $a_{N,n}|_{\partial G}=0$, while for $G=\mathbb{C}^d$ the assumption means exactly that each weight $a_{N,n}$ is rapidly decreasing at $\infty$ (cf.~\cite[remark previous to 1.2]{BBG}). In this setting (which we will in the sequel call the \textit{balanced setting}) we have $\overline{B_a^{\circ}}^{\text{co}}=B_a$ (cf.~\cite[1.5.(c)]{BBG}) and by \cite[1.6.(d)]{BBG}, $(AH)_0(G) \subseteq AH(G)$ is a topological subspace. Bierstedt, Bonet, Taskinen \cite[1.13]{BBT} showed that $\tilde{w}=\tilde{w}_0$ if $Ha_0(G)''=Ha(G)$ holds isometrically. By \cite[1.5.(d)]{BBG} the latter is the case in the balanced setting. One of the crucial techniques used by Bierstedt, Bonet, Galbis is based on the existence of a Taylor series representation about zero for each $f\in H(G)$,
$$
f(z)=\Bigsum{k=0}{\infty}p_k(z)\;\; \text{ for } z\in G,
$$
where $p_k$ is a $k$-homogeneous polynomial for $k\geqslant0$. The series converges to $f$ uniformly on each compact subset of $G$. The \textit{Ces\`{a}ro means} of the partial sums of the Taylor series of $f$ are denoted by $S_j(f)$, $j\geqslant0$, that is,
$$
[S_j(f)](z)={\textstyle\frac{1}{j+1}}\Bigsum{\ell=0}{j}\big(\Bigsum{k=0}{\ell}p_k(z)\big)\;\; \text{ for } z\in G.
$$
Each $S_j(f)$ is a polynomial of degree less or equal to $j$ and $S_j(f)\rightarrow f$ uniformly on every compact subset of $G$ (cf.~\cite[Section 1]{BBG}).
\smallskip
\\In the balanced setting it is now easy to verify that the projective spectrum $\mathcal{A}_0H$ is strongly reduced in the sense of \cite[3.3.5]{Wengenroth}. In fact it is a reduced projective limit in the sense of K\"othe \cite[p.~120]{KoetheI}, which is a stronger condition: By definition we have $\mathbb{P}\subseteq H(a_{N,n})_0(G)$ for all $N$ and $n$ and each $Ha_{N,n}(G)$ has a topology stronger than co. Thus, by Bierstedt, Bonet, Galbis \cite[1.6]{BBG} it follows that $\mathbb{P}$ is dense in $(\mathcal{A}_N)_0H(G)$ for each $N$. In particular, the polynomials are contained in the projective limit $(AH)_0(G)$ and hence
$$
(\mathcal{A}_N)_0H(G)\supseteq\overline{(AH)_0(G)}^{(\mathcal{A}_N)_0H(G)}\supseteq \overline{\mathbb{P}}^{(\mathcal{A}_n)_0H(G)}=(\mathcal{A}_N)_0H(G)
$$
holds for each $N$, i.e.~the projective limit is dense in every step. In fact, the space of all polynomials $\mathbb{P}$ on $G$ is also dense in the projective limit $(AH)_0(G)$.

\begin{thm}\label{o-growth-necessary-conditions}Assume that we are in the balanced setting. Then we have the implications (i)$\Rightarrow$(ii)$\Rightarrow$(iii)$\Rightarrow$(iv), where\vspace{2pt}
\\\begin{tabular}{rlrl}\vspace{1pt}
(i)  &$\Projeins \mathcal{A}_0H=0$,              & (iii) &$(AH)_0(G)$ is barrelled,                                           \\
(ii) &$(AH)_0(G)$ is ultrabornological,          & (iv)  &$\mathcal{A}$ satisfies condition $\text{(wQ)}^{\sim}_{\text{in}}$. \\
\end{tabular}
\end{thm}
\begin{proof} In view of Wengenroth \cite[3.3.4]{Wengenroth} (cf.~Vogt \cite[5.7]{VogtLectures}) it is enough to show that (iii) implies (iv): By \cite[3.3.6]{Wengenroth} (cf.~\cite[5.10]{VogtLectures}) barrelledness and reducedness imply condition $\text{(P}_{\!2}^{\star}\text{)}$, that is for each $N$ there exist $M$ and $n$ such that for each $K$ and $m$ there exist $k$ and $S>0$ such that for all $\varphi\in(\mathcal{A}_N)_{0}H(G)'$ the estimate $\|\varphi\|_{M,m}^{\star} \leqslant S(\|\varphi\|_{N,n}^{\star}+\|\varphi\|_{K,k}^{\star})$ holds, where $\|\varphi\|_{N,n}^{\star}=\sup_{f\in B_{N,n}^{\circ}}|\varphi(f)|$ denotes the dual norm. For arbitrary $z\in G$ we put $\varphi=\delta_z$ where $\delta_z(f):=f(z)$ and compute $\|\delta_z\|_{N,n}^{\star}=\sup_{\scriptscriptstyle f\in B_{N,n}^{\circ}}|\delta_z(f)|=(\frac{1}{a_{N,n}})_0^{\sim}(z)=(\frac{1}{a_{N,n}})^{\sim}(z)$.
Since the sum in the above condition can be estimated by two times the maximum we are done.
\end{proof}
\noindent{}In the proof above the reducedness of the projective spectrum $\mathcal{A}_0H$ is an essential ingredient. Due to the fact that we have no information about reducedness of $\mathcal{A}H$, it is not possible to proceed in the above way in order to get a corresponding result for O-growth conditions. However, a result for O-growth conditions, which is completely similar to \ref{o-growth-necessary-conditions}, is true. For its proof we need the next lemma, which is just an abstract formulation of a method invented by Bierstedt, Bonet \cite[Proof of \textquotedblleft{}(ii)$\Rightarrow$(iii)\textquotedblright{} of 3.10]{BierstedtBonet1988a} and which is also the key point in \cite[Proof of 3.8.(2)]{ABB2009}.

\begin{lem}\label{Abstract Lemma} Let $X$ and $X_0$ be locally convex spaces and $J\colon X_0\rightarrow X$ be a linear and continuous map. Assume that there exists an equicontinuous net $(S_{\alpha})_{\alpha\in A}\subseteq L(X,X_0)$ such that $S_{\alpha}(J(x))\rightarrow x$ holds for each $x\in X_0$. If $X$ is barrelled, then $X_0$ is quasibarrelled.
\end{lem}

\begin{proof} Let $T_0$ be a bornivorous barrel in $X_0$. We put $T:=\Bigcap{\alpha\in A}S_{\alpha}^{-1}(T_0)$. Since the $S_{\alpha}$ are linear and continuous, $T$ is has to be absolutely convex and closed. It is not hard to conclude that $T$ is also absorbing (and hence a barrel) utilizing that $(S_{\alpha})_{\alpha\in{}A}$ is equicontinuous. Since $X$ is barrelled, $T$ has to be a 0-neighborhood and using that $S_{\alpha}(J(x))\rightarrow x$ holds for each $x\in X_0$ it is easy to see that $J^{-1}(T)\subseteq T_0$ is valid which provides that $T_0$ is a 0-neighborhood.
\end{proof}
\noindent{}In the above setting, the mapping $J$ is nearly open in the sense of Pt\'{a}k (cf.~K\"othe \cite[p.~24]{KoetheII}), i.e.~for each 0-neighborhood $U$ in $X_0$, $\overline{J(U)}^{X}$ is a 0-neighborhood in $\overline{\im J}^X$. However, in all situations where we apply \ref{Abstract Lemma}, $J$ will turn out to be even an open mapping.

\begin{lem}\label{S_j equicontinuous net} Assume that we are in the balanced setting. The family $(S_j)_{j\in\mathbb{N}}$ of the Ces\`{a}ro means of the partial sums of the Taylor series is an equicontinuous net in the space $L(AH(G),(AH)_0(G))$ which satisfies $S_{j}(J(f))\rightarrow f$ for each $f\in (AH)_0(G)$, where $J\colon(AH)_0(G)\rightarrow AH(G)$ is the inclusion mapping.
\end{lem}
\begin{proof}By \cite[1.2.(b)]{BBG} the sequence $(S_j)_{j\in\mathbb{N}}\subseteq{}L(Ha_{N,n}(G),H(a_{N,n})_0(G))$ is an equicontinuous net for all $N$ and $n$.
\smallskip
\\We fix $N$ and claim that $(S_{j})_{j\in\mathbb{N}}\subseteq L(\mathcal{A}_NH(G),(\mathcal{A}_N)_0H(G))$ is equicontinuous. By Horv\'{a}th \cite[Proposition~3.4.5]{Horvath} it is enough to show that $(S_{j})_{j\in\mathbb{N}}\subseteq L(Ha_{N,n}(G),\linebreak{}(\mathcal{A}_N)_0H(G))$ is equicontinuous for each $n$. For fixed $n$ let $V\subseteq(\mathcal{A}_N)_0H(G)$ be a 0-neighborhood. Then $V\cap H(a_{N,n})_0(G)$ is a 0-neighborhood in $H(a_{N,n})_0(G)$. By our assumptions there exists a 0-neighborhood $U$ in $Ha_{N,n}(G)$ such that $S_{j}(U)\subseteq V\cap H(a_{N,n})_0(G)\subseteq V$ for each $j$, which establishes the claim.
\smallskip
\\Now let $V$ be a 0-neighborhood in $(AH)_0(G)$. Then there exist $N$ and a 0-neighborhood $V'$ in $(\mathcal{A}_N)_0H(G)$ such that $V=V'\cap (AH)_0(G)$. By the above there exists a 0-neighborhood $U'$ in $(\mathcal{A}_N)_0H(G)$ such that $S_{j}(U')\subseteq V'$ for each $j$. We put $U:=U'\cap AH(G)$, which is a 0-neighborhood in $AH(G)$ and obtain $S_{j}(U)=S_{j}(U'\cap AH(G))\subseteq V'\cap (AH)_0(G)=V$ for each $j$ and have shown that $(S_j)_{j\in\mathbb{N}}\subseteq L(AH(G),(AH)_0(G))$ is equicontinuous.
\smallskip
\\Let finally $f\in(AH)_0(G)$ and $N$ be arbitrary. Then there exists $n$ such that $f\in H(a_{N,n})_0H(G)$. Since $S_jf$ is a polynomial, we have $S_jf\!\in\! H(a_{N,n})_0H(G)$. By \cite[1.2.(e)]{BBG}, $S_jf\rightarrow f$ holds in $H(a_{N,n})_0H(G)$ and thus also in $(\mathcal{A}_N)_0H(G)$. Because $N$ was arbitrary, we obtain $S_jf\rightarrow f$ in $(AH)_0(G)$.
\end{proof}

\begin{cor}\label{barrelled-barrelled} Assume that we are in the balanced setting and assume $AH(G)$ to be barrelled. Then $(AH)_0(G)$ is barrelled.
\end{cor}
\begin{proof}\ref{Abstract Lemma} and \ref{S_j equicontinuous net} provide that $(AH)_0(G)$ is quasibarrelled. By Vogt \cite[3.1]{Vogt1989} for $(AH)_0(G)$ quasibarrelledness is equivalent to barrelledness, since the projective spectrum in the o-growth case is reduced.
\end{proof}

\begin{thm}\label{strong necessary condition} Assume that we are in the balanced setting. Then we have the implications (i)$\Rightarrow$(ii)$\Rightarrow$(iii)$\Rightarrow$(iv), where\vspace{3pt}
\\\begin{tabular}{rlrl}\vspace{1pt}
(i)   & $\Projeins \mathcal{A}H=0$,      & (iii) & $AH(G)$ is barrelled, \\
(ii)  & $AH(G)$ is ultrabornological,    & (iv)  &$\mathcal{A}$ satisfies condition $\text{(wQ)}^{\sim}_{\text{in}}$. \\
\end{tabular}
\end{thm}
\begin{proof}The result follows directly from \ref{o-growth-necessary-conditions} and \ref{barrelled-barrelled}.
\end{proof}


\section{\hspace{-10pt}Sufficient conditions for ultrabornologicity and barrelledness: Results for the unit disc}\label{Sufficient}

To find sufficient conditions for the vanishing of $\Projeins \mathcal{A}H$ and for barrelledness of $(AH)_0(G)$, we need to decompose holomorphic functions. In the case of the unit disc, a decomposition suitable for our purposes is possible if we assume that our defining sequence $\mathcal{A}$ belongs to some set of weights $W$ which is assumed to be of \textit{class $\mathcal{W}$} defined by Bierstedt, Bonet \cite{BB2003}. That is, we assume that $W$ consists of radial weights and further that each $w\in W$ satisfies $\lim_{r\nearrow 1}w(r)=0$ and is non-increasing if restricted to $[0,1[$. We assume $W$ to be stable under multiplication with strictly positive scalars and under the formation of finite minima. Next, we assume that there exists a sequence of linear and continuous operators $(R_n)_{n\geqslant1}$, $R_n\colon (H(\mathbb{D}),\text{co})\rightarrow (H(\mathbb{D}),\text{co})$ such that for $n\geqslant1$ the image of $R_n$ is a finite dimensional subspace of the space $\mathbb{P}$ of polynomials on 
$\mathbb{D}$. Further we assume that for each $p\in \mathbb{P}$ there exists $n$ with $R_np=p$ and that $R_n\circ R_m =R_{\min(n,m)}$ holds for arbitrary $n,\,m\geqslant1$. Moreover, we require that there is $c>0$ such that for each $n\geqslant1$, $r\in\,]0,1[$ and $p\in\mathbb{P}$ the estimate $\sup_{|z|=r}|[R_np](z)|\leqslant c\sup_{|z|=r}|p(z)|$ holds. By setting $R_0:=0$ and $r_n:=1-2^{-n}$ for $n\geqslant0$ we get a system $(R_n,r_n)_{n\geqslant0}$ which is assumed to satisfy the following two conditions.
\begin{itemize}
\item[(P1)]
$
\exists \: C\geqslant 1\; \forall \: v\in W,\,p\in \mathbb{P}\colon
$
$$
{\textstyle\frac{1}{C}} \sup_{n\in\mathbb{N}}\big(v(r_n)\sup_{| z|=r_n}\big| [(R_{n+2}-R_{n-1})p](z)\big|\big) \leqslant \sup_{z\in\mathbb{D}} v(z)| p(z)|,
$$
$$
\sup_{z\in\mathbb{D}}v(z)| p(z)|\leqslant C \sup_{n\in\mathbb{N}}\big(v(r_n)\sup_{| z| = r_n}\big|[(R_{n+1}-R_n)p](z)\big|\big).
$$
\item[(P2)]
$
\forall \: v\in W\; \exists \: D(v)\geqslant 1 \;\forall \: (p_n)_{n\in\mathbb{N}}\subseteq\mathbb{P},\, p_n\not=0
$
only for finitely many $n\colon$
$$
\sup_{z\in\mathbb{D}} v(z)\big|\Bigsum{n=1}{\infty}[(R_{n+1}-R_n)p_n](z)\big|\leqslant D(v)\sup_{k\in\mathbb{N}}\big(v(r_k)\sup_{|z|=r_k}| p_k(z)|\big).
$$
\end{itemize}
Note that for a system of weights in $W$ the requirements of the balanced setting are automatically satisfied. Moreover, Theorem's 3.1 and 4.1 of Bierstedt, Bonet \cite{BB2003} and the results of Bierstedt, Meise, Summers \cite{BMS1982} imply that for $\mathcal{A}\subseteq W$, $\mathcal{A}_NH(\mathbb{D})\subseteq\mathcal{A}_NC(\mathbb{D})$ and $(\mathcal{A}_N)_0H(\mathbb{D})\subseteq (\mathcal{A}_N)_0C(\mathbb{D})$ are all topological subspaces. Here, $\mathcal{A}_NC(\mathbb{D})$ and $(\mathcal{A}_N)_0C(\mathbb{D})$ denote weighted LB-spaces of continuous functions which are defined analogously to $\mathcal{A}_NH(\mathbb{D})$ and $(\mathcal{A}_N)_0H(\mathbb{D})$ but \textquotedblleft{}holomorphic\textquotedblright{} replaced with \textquotedblleft{}continuous\textquotedblright{}.
\medskip
\\Our investigation in Section \ref{Necessary} has shown that concerning necessary conditions for the vanishing of $\Projeins$ the o-growth case was easier to handle than the O-growth case, since the spectrum $\mathcal{A}_0H$ is reduced. However, also in the O-growth case the balanced setting allowed to prove \textquotedblleft{}the same\textquotedblright{} result: In both situations, $\text{(wQ)}^{\sim}_{\text{in}}$ is necessary for barrelledness. As we will see in the sequel for sufficient conditions the situation is the other way round, that is the O-growth case is the easier one. But the situation is not symmetric: We are not able to prove sufficient conditions for $\Projeins \mathcal{A}_0H=0$ at all.

\begin{thm}\label{Q-out-sim implies Projeins Null}Let $\mathcal{A}$ be a sequence in $W$ and assume that $\mathcal{A}$ satisfies condition $\text{(Q)}^{\sim}_{\text{out}}$. Then $\Projeins \mathcal{A}H=0$.
\end{thm}
\begin{proof}
In order to show that $\Projeins \mathcal{A}H=0$ we use Braun, Vogt \cite[Theorem 8]{BraunVogt1997} (which was independently obtained by Frerick, Wengenroth \cite{FrerickWengenroth1996}). That is, we have to show condition $\text{(}\overline{\text{P}_{\!2}}\text{)}$
$$
\forall \: N \; \exists \: M,\,n \; \forall \: K,\,m,\,\varepsilon>0 \; \exists \: k,\,S>0\: \colon \: B_{M,m}\subseteq \varepsilon B_{N,n}+SB_{K,k}.
$$
For given $N$ we select $M$ and $n$ as in $\text{(Q)}^{\sim}_{\text{out}}$. For given $K,\,m,\,\varepsilon>0$ we put $\varepsilon':=\frac{\varepsilon}{(D_1+2c^2)C}$ and choose $K$ and $S'>0$ according to $\text{(Q)}^{\sim}_{\text{out}}$ w.r.t.~$\varepsilon'$ and put $S:=S'(2c^2+D_2)$. Now we fix an arbitrary $f\in B_{M,m}$ and consider $S_tf$. We have $a_{M,m}|S_tf|\leqslant a_{M,m}|f|\leqslant 1$, i.e.~$|S_tf|\leqslant \frac{1}{a_{M,m}}$. With \cite[1.2.(iii)]{BBT} it follows $|S_tf|\leqslant (\frac{1}{a_{M,m}})^{\sim}$ and by the estimate in $\text{(Q)}^{\sim}_{\text{out}}$ we obtain $|S_tf|\leqslant\max(\varepsilon'(\frac{1}{a_{N,n}}),\,S'(\frac{1}{a_{K,k}}))^{\sim}\leqslant\max(\frac{\varepsilon'}{a_{N,n}},\frac{S'}{a_{K,k}})$ where the last estimate follows from \cite[1.2.(i)]{BBT}. We put $u_1:=\frac{a_{N,n}}{\varepsilon'}$, $u_2:=\frac{a_{K,k}}{S'}$ and $u:=\min(u_1,\,u_2)$. Then the above transforms to $|S_tf|\leqslant \max(\frac{1}{u_1},\frac{1}{u_2})=\frac{1}{u}$, i.e.~$u|S_tf|\leqslant 1$. As $W$ 
is closed under the formation of finite minima and under multiplication with positive scalars, $u\in W$ holds.
\smallskip
\\From now on we use the decomposition method invented by Bierstedt, Bonet \cite{BB2003}: We decompose $S_tf=R_1S_tf+\sum_{\nu=1}^{\infty}(R_{\nu+1}-R_{\nu})S_tf$ where both summands are polynomials.
\smallskip
\\Let us study the first summand: By the estimate previous to (P1) there exists $c>0$ such that $\sup_{|z|=r_1}|[R_1S_tf](z)|\leqslant c \sup_{|z|=r_1}|S_tf(z)|$. We multiply with $u(r_1)$ and use $u|S_tf|\leqslant 1$ to get $u(r_1)\sup_{|z|=r_1}|[R_1S_tf](z)|\leqslant c u(r_1) \sup_{|z|=r_1}|S_tf(z)|\leqslant 1$. By the definition of $u$ we obtain $u(r_1)=\min(u_1(r_1), u_2(r_2))$. Let $i\in\{1,\,2\}$ such that $u(r_1)=u_i(r_1)$. Now we use the second inequality of (P1) to get $C\geqslant 1$ such that
\begin{align*}
\sup_{z\in\mathbb{D}}u_i(z)|R_1S_tf(z)|\;& \leqslant  \;C\sup_{n\in\mathbb{N}}\big(u_i(r_n)\sup_{|z|=r_n}|[(R_{n+1}-R_n)R_1S_tf](z)|\big)\\
& =          \;C u(r_1)\sup_{|z|=r_1}|[(R_2-R_1)R_1S_tf](z)| \\
& \leqslant  \;2cC u(r_1)\sup_{|z|=r_1}|R_1S_tf(z)| \\
&\leqslant   \;2c^2C.
\end{align*}
By the definition of the $u_i$ and the choice of $i$ we get $\sup_{z\in\mathbb{D}}a_{N,n}(z)|R_1S_tf(z)|\linebreak\leqslant 2c^2C\varepsilon'$ or $\sup_{z\in\mathbb{D}}a_{K,k}(z)|R_1S_tf(z)|\leqslant 2c^2CS$, i.e.~$R_1S_tf\in 2c^2C\varepsilon' B_{N,n}$ or $R_1S_tf\in 2c^2CS B_{K,k}$.
\smallskip
\\Now we consider $S_tf-R_1S_tf=\sum_{\nu=1}^{\infty}(R_{\nu+1}-R_{\nu})S_tf$. We use the first inequality of (P1) for $u$ and $S_tf$ to get with the same $C\geqslant 1$ as above that $\frac{1}{C}\sup_{\nu\in\mathbb{N}}(u(r_{\nu})\sup_{|z|=r_{\nu}}|[(R_{\nu+2}-R_{\nu-1})S_tf](z)|)\leqslant \sup_{z\in\mathbb{D}}u(z)|S_tf(z)|\leqslant 1$ holds, i.e.~for each $\nu$ we have $u(r_{\nu})\sup_{|z|=r_{\nu}}\linebreak{}|[(R_{\nu+2}-R_{\nu-1})S_tf](z)|\leqslant C$. Now we write $\mathbb{N}=J_1\,\dot{\cup}\,J_2$ such that $u(r_j)= u_1(r_j)$ for $j\in J_1$ and $u(r_j)=u_2(r_j)$ for $j\in J_2$. For $i\in\{1,\,2\}$ we put
$$
g_i:=\Bigsum{\nu\in J_i}{}(R_{\nu+1}-R_n)S_tf \;\; \text{ and } \;\; p_{\nu}^i:=\begin{cases} \;(R_{\nu+2}-R_{\nu-1})S_tf & \text{for } n\in J_i\\ \;\;\;\;\;\;\;\;\;\;\;\;\;0 & \text{otherwise.} \end{cases}
$$
Then we obtain $S_tf-R_1S_tf=g_1+g_2$ by construction and the properties of class $\mathcal{W}$ yield $g_i=\sum_{\nu\in J_i}(R_{\nu+1}-R_{\nu})S_tf=\sum_{\nu\in J_i}(R_{\nu+1}-R_{\nu})(R_{\nu+2}-R_{\nu-1})S_tf=\sum_{\nu\in J_i}(R_{\nu+1}-R_\nu)p_{\nu}^i$. Since $(p_{\nu}^i)_{\nu\in\mathbb{N}}\subseteq\mathbb{P}$ with only finitely many $p_{\nu}^i\not=0$ we can apply (P2) and get $D(u_i)=:D_i\geqslant 1$ such that
\begin{align*}
\sup_{z\in\mathbb{D}}u_i(z)|g_i(z)|\;& =          \;\sup_{z\in\mathbb{D}}\big|\Bigsum{\nu=1}{\infty}[(R_{\nu+1}-R_{\nu})p_{\nu}^i](z)\big| \\
                                     & \leqslant  \;D_i \sup_{\nu\in\mathbb{N}}\big(u_i(r_{\nu}) \sup_{|z|=r_{\nu}}|p_{\nu}^i(z)|\big)\\
			             & \leqslant  \;D_i \sup_{\nu\in J_i}\big(u(r_{\nu})\sup_{|z|=r_{\nu}}|[(R_{\nu+2}-R_{\nu-1})S_tf](z)|\big)\\
			             & \leqslant  \;D_iC.
\end{align*}
This yields $g_1\in D_1C\varepsilon' B_{N,n}$ and $g_2\in D_2CS'B_{K,k}$. Thus, $S_tf=R_1S_tf+g_1+g_2\in \varepsilon'(2c^2+D_1)CB_{N,n}+S'(2c^2+D_2)B_{K,k}=\varepsilon B_{N,n}+SB_{K,k}$ for arbitrary $t$. Since $B_{N,n}$ and $B_{K,k}$ are both co-compact and $S_tf\rightarrow f$ holds w.r.t.~co we obtain $f\in \varepsilon B_{N,n}+SB_{K,k}$ and hence $\text{(}\overline{\text{P}_{\!2}}\text{)}$.
\end{proof}
\noindent{}We noted already in Section \ref{Preliminaries} that (Q) and $\text{(Q)}^{\sim}_{\text{in}}$ both imply $\text{(Q)}^{\sim}_{\text{out}}$. Thus, in \ref{Q-out-sim implies Projeins Null} it is possible to replace $\text{(Q)}^{\sim}_{\text{out}}$ by $\text{(Q)}^{\sim}_{\text{in}}$ or (Q). However, $\text{(Q)}^{\sim}_{\text{out}}$ a priori is a weaker (albeit less accessible) condition than (Q). Before we start with the preparations for our next result, let us precise the remarks we made previous to \ref{Q-out-sim implies Projeins Null}: In the proof of the forementioned result we used in the final step that the balls $B_{N,n}$ are co-compact. Unfortunately, for the balls $B_{N,n}^{\circ}$ this cannot be true: If $B_{N,n}^{\circ}$ is co-closed for all $N$ and $n$, then we get $B_{N,n}^{\circ}=\overline{B_{N,n}^{\circ}}^{\text{co}}=B_{N,n}$ where the last set is co-compact. The equality $B_{N,n}^{\circ}=B_{N,n}$ yields $Ha_{N,n}(G)=H(a_{N,n})_0(G)$ which implies (using $H(a_{N,n})_0(G)''=Ha_{N,n}(G)$, 
see \cite[1.5.(d)]{BBG}) that $H(a_{N,n})_0(G)$ is reflexive. By Bonet, Wolf \cite[Corollary 2]{BW2003} this implies that the space is finite dimensional -- but already in the balanced setting and in particular if $\mathcal{A}\subseteq W$ we have $\mathbb{P}\subseteq H(a_{N,n})_0(G)$. Thus, it is not possible to get a result for o-growth conditions by proceeding analogously to the proof of \ref{Q-out-sim implies Projeins Null}. However, utilizing results established in Bonet, Wegner \cite{BonetWegner2010} we can find a sufficient condition for $(AH)_0(\mathbb{D})$ being barrelled under the assumptions of class $\mathcal{W}$. Let us fix $N$, i.e.~we fix a step in the projective spectrum $\mathcal{A}_0H$. Now we consider the space of polynomials $\mathbb{P}$ endowed with two a priori different topologies: We write $\mathbb{P}$ for this space endowed with the topology induced by $(\mathcal{A}_N)_0H(\mathbb{D})$ and put $(\mathcal{A}_N)_0P(\mathbb{D}):=\ind{n}P(a_{N,n})_0(\mathbb{D})$ where $P(a_{N,n})_0=(\mathbb{P},\|\cdot\|_{N,n})$. With the techniques used in the proof of \cite[3.1]{BB2003} it is not hard to verify that the two topologies which we defined on $\mathbb{P}$ coincide, i.e.~$(\mathcal{A}_N)_0P(\mathbb{D})\subseteq (\mathcal{A}_N)_0H(\mathbb{D})\subseteq \mathcal{A}_NH(\mathbb{D})$ are topological subspaces for each $N$. Since $\mathcal{A}_NH(\mathbb{D})$ is regular, $(B_{N,n})_{n\in\mathbb{N}}$ is a fundamental system of bounded sets in the latter space, whence the balls $P_{N,n}^{\circ}=B_{N,n}\cap\mathbb{P}=B^{\circ}_{N,n}\cap\mathbb{P}$ form a fundamental system of bounded sets in $(\mathcal{A}_N)_0P(\mathbb{D})$.
\smallskip
\\Now we consider the projective spectrum $\mathcal{A}_0P=((\mathcal{A}_N)_0P(\mathbb{D}),\subseteq_{N+1}^N)_{N\in\mathbb{N}}$ with inclusions as linking maps. From \cite[2.2]{BonetWegner2010} (where general projective limits of regular inductive limits of normed spaces with inclusions as linking maps are studied) it follows that $(\mathcal{A}_N)_0P(\mathbb{D})$ is bornological, if it satisfies condition (B1), i.e.~for each $N$ there exists $M$ such that for each $m$ there exists $n$ such that $P^{\circ}_{M,m}\subseteq \Bigcap{k\in\mathbb{N}}(P^{\circ}_{N,n}\cap X + {\textstyle\frac{1}{k}}P^{\circ}_{N,n})$ and in addition for each absolutely convex set $T\subseteq X$ condition (B2) holds, that is there exists $N$ such that for each $n$ there exists $S>0$ such that $P^{\circ}_{N,n}\cap X\subseteq ST$. In order to check condition (B2) we have for technical reasons to assume that our set $W$ of class $\mathcal{W}$ is closed under finite maxima.

\begin{prop}\label{polynomials barrelled-lemma} Let $\mathcal{A}\subseteq W$ and assume that $W$ is closed under finite maxima. Let $\mathcal{A}$ satisfy condition (wQ). Then the space $(AP)_0(G):=\proj{N}\ind{n}P(a_{N,n})_0(\mathbb{D})$ is bornological.
\end{prop}
\begin{proof} By Bierstedt, Bonet \cite{BB1994}, condition (wQ) implies condition $\text{(wQ)}^{\star}$ that is
$$
\exists\: (n(\sigma))_{\sigma\in\mathbb{N}}\subseteq \mathbb{N} \text{ increasing } \forall\:N\;\exists\:M\;\forall\:K,\,m\:\exists\:S>0,\,k\colon 
$$
$$
{\textstyle\frac{1}{a_{M,m}}}\leqslant S\max\big({\textstyle\frac{1}{a_{K,k}}},\min_{\scriptscriptstyle\sigma=1,\dots,N}{\textstyle\frac{1}{a_{\sigma,n(\sigma)}}}\big).
$$
We fix an absolutely convex and bornivorous set $T$ in $(AP)_0(\mathbb{D})$.  Since $(AP)_0(\mathbb{D})=P(a_{N,n})_0(\mathbb{D})$ holds algebraically for all $N$, $n$ we may consider $T$ as a subset of the latter space and claim that there exists $N$ such that for each $n$ the ball $P^{\circ}_{N,n}$ is absorbed by $T$. We proceed by contradiction and hence assume that for each $M$ there exists $m(M)$ such that $P^{\circ}_{M,m(M)}$ is not absorbed by $T$. By \cite[3.1]{BonetWegner2010} there exists $N$ such that $\Bigcap{\sigma=1}^{N}P_{\sigma,m(\sigma)}^{\circ}$ is absorbed by $T$. For the sequence $(n(\sigma))_{\sigma\in\mathbb{N}}$ and this $N$ we choose $M$ as in $\text{(wQ)}^{\star}$. By our assumption there exists $m(M)$ such that for each $K$ there exist $S_K>0$ and $k(K)$ such that $\frac{1}{a_{M,m(M)}}\leqslant S_K\max({\textstyle\frac{1}{a_{K,k(K)}}},\min_{\sigma=1,\dots,N}\frac{1}{a_{\sigma,n(\sigma)}})$. We claim that the estimate ${\textstyle\frac{1}{a_{M,m(M)}}}\leqslant S_K'\max(u_K,w_N)$ holds 
for each $K$, where we use the abbreviations $w_N:=\min_{\sigma=1,\dots,N}\frac{1}{a_{\sigma,n(\sigma)}}$, $u_K:=\min_{\mu=1,\dots,K}\frac{1}{a_{\mu,k(\mu)}}$ and $S_K':=\max_{\mu=1,\dots,K}S_{\mu}$. To establish the claim let us fix $K$. Then we have $\frac{1}{S_K'a_{M,m(M)}}\leqslant \frac{1}{S_{\mu}a_{M,m(M)}}\leqslant \max(\frac{1}{a_{\mu,k(\mu)}},w_N)$ for $\mu=1,\dots,K$ by the very definition of $S_K'$ and the estimate we deduced from $\text{(wQ)}^{\star}$. If now $\frac{1}{S_K'a_{M,m(M)}}\leqslant w_N$ holds, we are done. Otherwise the above yields $\frac{1}{S_K'a_{M,m(M)}}\leqslant \frac{1}{a_{\mu,k(\mu)}}$ for $\mu=1,\dots,K$, i.e.~$\frac{1}{S_K'a_{M,m(M)}}\leqslant \min_{\mu=1,\dots,K}\frac{1}{a_{\mu,k(\mu)}}=u_K$ and we are done as well.
\smallskip
\\Now we again make use of the decomposition method based on class $\mathcal{W}$ to show that for each $K$ there exists $\tau_K>0$ such that the inclusion $P_{M,m(M)}^{\circ}\subseteq \tau_K[\Bigcap{\sigma=1}^N P_{\sigma,n(\sigma)}^{\circ}+\Bigcap{\mu=1}^KP_{\mu,k(\mu)}^{\circ}]$ is valid. Let $K$ be fixed and $p\in P_{M,m(M)}^{\circ}$ be given, i.e.~$|p|\leqslant\frac{1}{a_{M,m(M)}}$. We get the estimate
$
|p|\leqslant S_K'\max(u_K,w_N)=\max(\min_{\sigma=1,\dots,N}\frac{S_K'}{a_{\sigma,n(\sigma)}},\min_{\mu=1,\dots,K}\frac{S_K'}{a_{\mu,k(\mu)}}),
$
may define $\frac{1}{u_1}:=\min_{\sigma=1,\dots,N}\frac{S_K'}{a_{\sigma,n(\sigma)}}$, $\frac{1}{u_2}:=\min_{\mu=1,\dots,K}\frac{S_K'}{a_{\mu,k(\mu)}}$ and thus obtain
$u_1:=\max_{\sigma=1,\dots,N}\frac{a_{\sigma,n(\sigma)}}{S_K'}$, $u_2:=\max_{\mu=1,\dots,K}\frac{a_{\mu,k(\mu)}}{S_K'}\in W$ since $W$ is closed under the formation of finite maxima. We put $u:=\min(u_1,u_2)$. Since $W$ is closed under finite minima, $u\in W$ holds. Moreover, $\frac{1}{u}=\max(\frac{1}{u_1},\frac{1}{u_2})$ that is by the above $u|p|\leqslant 1$. 
\smallskip
\\Now we repeat the proof of \ref{Q-out-sim implies Projeins Null} in order to obtain $p=R_1p+g_1+g_2\in (2c^2+D_1)\linebreak{}CS_K'\Bigcap{\sigma=1}^NP_{\sigma,n(\sigma)}^{\circ} + (2c^2+D_2)CS_K'\Bigcap{\mu=1}^KP_{\mu,k(\mu)}^{\circ}$, i.e.~$p\in\tau_K\big[\Bigcap{\sigma=1}^NP_{\sigma,n(\sigma)}^{\circ}+\Bigcap{\mu=1}^KP_{\mu,k(\mu)}^{\circ}\big]$ with $\tau_K=CS_K'(2c^2+\max(D_1,D_2))$, which establishes the claim.
\smallskip
\\By \cite[3.1]{BonetWegner2010} there exists $K'$ such that $\Bigcap{\mu=1}^{K'} P^{\circ}_{\mu,k(\mu)}$ is absorbed by $T$, which in turn implies
$$
P_{M,m(M)}^{\circ}\subseteq \tau_{K'}\big[\Bigcap{\sigma=1}^N P_{\sigma,n(\sigma)}^{\circ}+\Bigcap{\mu=1}^{K'}P_{\mu,k(\mu)}^{\circ}\big],
$$
where the set on the left hand side is not absorbed by $T$ unlike the set on the right hand side, a contradiction. To finish the proof, we observe that our claim is exactly the statement (B2) in \cite[Section 2]{BonetWegner2010}, cf.~our remarks previous to \ref{polynomials barrelled-lemma}. Since statement (B1) of \cite{BonetWegner2010} is trivial in the case of $(AP)_0(\mathbb{D})$, \cite[2.2]{BonetWegner2010} finishes the proof.
\end{proof}

\begin{thm}\label{A_0P(G) bornological implies A_0H(G) barrelled} Let $\mathcal{A}\subseteq W$ and assume that $W$ is closed under finite maxima. If $\mathcal{A}$ satisfies condition (wQ), then $(AH)_0(\mathbb{D})$ is barrelled.
\end{thm}
\begin{proof} By \ref{polynomials barrelled-lemma}, the space $(AP)_0(\mathbb{D})$ is bornological. Moreover, $(AP)_0(\mathbb{D})\subseteq(AH)_0(\mathbb{D})$ is a topological subspace and this subspace is dense as we noted in Section \ref{Necessary}. From these facts it is easy to derive that $(AH)_0(\mathbb{D})$ is quasibarrelled. As we mentioned already in the proof of \ref{barrelled-barrelled}, for $(AH)_0(G)$, quasibarrelledness is equivalent to barrelledness by Vogt \cite[3.1]{Vogt1989}, since $\mathcal{A}_0H$ is reduced.
\end{proof}

\begin{rem}\label{schemes} The results of Sections \ref{Necessary} and \ref{Sufficient} are summarized in Figs.~1 and 2, which in particular illustrate again the lines of the proofs for the results and the assumptions needed for each single implication.
\begin{center}
\psfig{figure=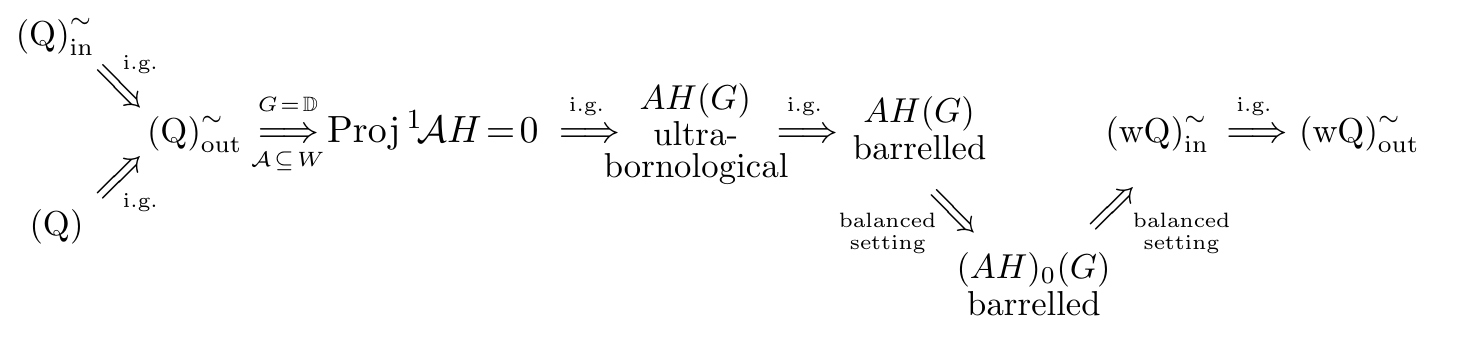,width=12.85cm}
{\small Figure 1: Scheme of implications for O-growth conditions.}
\bigskip
\\\psfig{figure=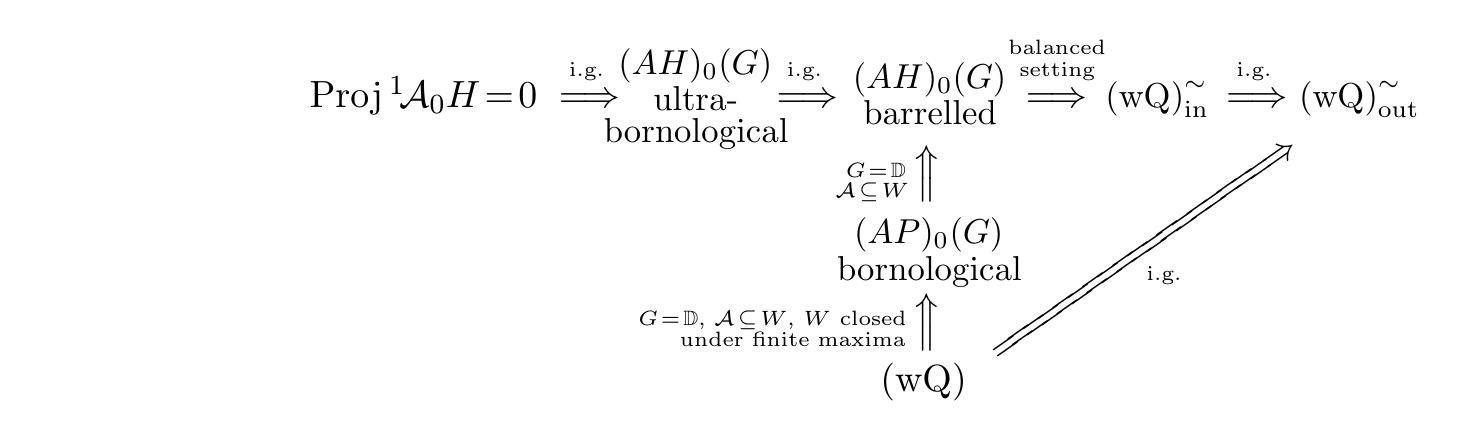,width=12.85cm}
{\small Figure 2: Scheme of implications for o-growth conditions.}
\end{center}
\end{rem}


\section{\hspace{-10pt}Interchangeability of projective and inductive limit}\label{Inter}

Given a sequence of weights $\mathcal{A}=((a_{N,n})_{n\in\mathbb{N}})_{N\in\mathbb{N}}$ on an open set $G\subseteq \mathbb{C}^d$ we can -- in addition to the PLB-spaces investigated in the preceeding sections -- also associate \textit{weighted LF-spaces of holomorphic functions} by defining
\begin{align*}
\mathcal{V}H(G)&:=\ind{n}\proj{N}Ha_{N,n}(G),\\
\mathcal{V}_0H(G)&:=\ind{n}\proj{N}H(a_{N,n})_0(G).
\end{align*}
These spaces constitute the holomorphic version of the weighted LF-spaces of continuous functions investigated by Bierstedt, Bonet \cite{BB1994} and have been studied by several authors in different contexts, see for instance Bierstedt, Meise \cite{BM2001} or Bierstedt, Bonet \cite{BB2003}. We refer to the survey article of Bierstedt \cite{Bierstedt2001} for detailed references. As in the case of continuous functions (cf.~\cite[p.~393f]{ABB2009}), it is clear that $\mathcal{V}H(G)\subseteq AH(G)$ and $\mathcal{V}_0H(G)\subseteq(AH)_0(G)$ holds with continuous inclusion and our aim is to investigate when the equality is valid. In order to do this we make use of the conditions (B) and $\text{(B)}^{\sim}$ which we introduced in Section \ref{Preliminaries} and of the results of the last two sections. We start with the investigation of the algebraic equalities $AH(G)=\mathcal{V}H(G)$ and $(AH)_0(G)=\mathcal{V}_0H(G)$, for which we need the following lemma.

\begin{lem}\label{B-tilde iff A=V alg-lemma} Let $G$ be balanced, $F\subseteq H(G)$ be a linear subspace which contains the polynomials and let $v$ and $w$ be two radial weights on $G$. If there exists $c>0$ such that $\sup_{z\in G} v(z)|f(z)|\leqslant c\,\sup_{z\in G} w(z)|f(z)|$ holds for each $f\in F$, then we have $\tilde{v}\leqslant c\tilde{w}$ on $G$.
\end{lem}
\begin{proof} For $g\in H(G)$ with $|g|\leqslant \frac{1}{w}$ we have $|S_tg|\leqslant\frac{1}{w}$. Since $S_tg\in\mathbb{P}\subseteq F$, we may apply the estimate in the lemma to obtain $\frac{1}{c}\:\sup_{z\in G} v(z)|S_tg(z)|\leqslant\sup_{z\in G} w(z)|S_t(z)|\leqslant 1$, i.e.~$|S_tg|\leqslant \frac{c}{v}$. Because $S_tg\rightarrow g$ converges pointwise, we get $|g|\leqslant \frac{c}{v}$, hence $|\frac{g}{c}|\leqslant \frac{1}{v}$ and since $\frac{g}{c}\in H(G)$, $|\frac{g}{c}|\leqslant \frac{1}{\tilde{v}}$ holds. Finally we have
$$
{\textstyle\frac{1}{\tilde{w}(z)}}\:=\!\!\!\!\sup_{\stackrel{g\in H(G)}{\scriptscriptstyle |g|\leqslant\frac{1}{w}\text{ on } G}}|g(z)|\leqslant \!\!\sup_{\stackrel{g\in H(G)}{\scriptscriptstyle |g|\leqslant\frac{c}{\tilde{v}}\text{ on } G}}|g(z)|\leqslant {\textstyle\frac{c}{\tilde{v}(z)}}
$$
for arbitrary $z\in G$.
\end{proof}

\begin{prop}\label{O-algebraic-equality}Assume that we are in the balanced setting. Then the equality $AH(G)=\mathcal{V}H(G)$ holds algebraically if and only if $\mathcal{A}$ satisfies condition $\text{(B)}^{\sim}$.
\end{prop}
\begin{proof}\textquotedblleft{}$\Rightarrow$\textquotedblright{}\: For a given sequence $(n(N))_{N\in\mathbb{N}}$ we consider $F:=\Bigcap{N\in\mathbb{N}}Ha_{N,n(N)}(G)$ endowed with the topology given by the system $(p_L)_{L\in\mathbb{N}}$ of seminorms $p_L(f)=\max_{N=1,\dots,L}\linebreak{}\sup_{z\in G} a_{N,n(N)}(z)|f(z)|$. Then $F$ is contained in $AH(G)$ with continuous inclusion and $AH(G)$ is complete and has a topology finer than co, whence $F$ is a Fr\'{e}chet space. The equality $AH(G)=\mathcal{V}H(G)$ implies that $F$ is contained in the LF-space $\mathcal{V}H(G)$. It is easy to see that the corresponding inclusion map has closed graph and with de Wilde's closed graph theorem (e.g.~\cite[24.31]{MeiseVogtEnglisch}) we get that it is even continuous. We apply Grothendieck's factorization theorem (e.g.~\cite[24.33]{MeiseVogtEnglisch}) to obtain $m$ such that $F\subseteq HV_m(G)$ holds with continuous inclusion. Hence, for given $M$ there exists $L$ and $c>0$ such that for each $f\in F$ the estimate $\sup_{z\in G}a_{M,m}(z)|f(z)|\leqslant{}c\max_{N=1,\dots,L} \sup_{z\in G} a_{N,n(N)}(z)|f(z)|\leqslant{}c\sup_{z\in G}\,(\max_{N=1,\dots,L} a_{N,n(N)})(z)|f(z)|$ holds. Since we are in the balanced setting we have $\mathbb{P}\subseteq F$ and we can apply \ref{B-tilde iff A=V alg-lemma} in order to obtain the estimate $\tilde{a}_{M,m}\leqslant c\,(\max_{N=1,\dots,L}a_{N,n(N)})^{\sim}$.
\smallskip
\\\textquotedblleft{}$\Leftarrow$\textquotedblright{}\: Let $f\in AH(G)$ be given. By definition, for each $N$ there exist $n(N)$ and $b_n>0$ such that $a_{N,n(N)}|f|\leqslant b_n$ holds on $G$. We select $m$ according to $\text{(B)}^{\sim}$ w.r.t.~the sequence $(n(N))_{N\in\mathbb{N}}$ and claim that $f\in Ha_{M,m}(G)$ holds for each $M$. Given $M$ we select $L$ and $c>0$ as in $\text{(B)}^{\sim}$ and put $b:=\max(b_1,\dots,b_L)$. Then we have $a_{N,n(N)}|\frac{f}{b}|\leqslant 1$ for $N=1,\dots,L$ and thus $\max_{N=1,\dots,L}a_{N,n(N)}|\frac{f}{b}|\leqslant1$ on $G$. We put $w_N:=\max_{N=1,\dots,L}a_{N,n(N)}$ to obtain $|\frac{f}{b}|\leqslant \frac{1}{w_N}$ and since $\frac{f}{b}\in H(G)$, we get $|\frac{f}{b}|\leqslant\frac{1}{\tilde{w}_N}$. $\text{(B)}^{\sim}$ implies $\frac{1}{\tilde{w}_N}\leqslant \frac{c}{\tilde{a}_{M,m}}$ and hence $|\frac{f}{b}|\leqslant\frac{c}{\tilde{a}_{M,m}}$ holds. This finally yields $\tilde{a}_{M,m}|f|\leqslant cb$, i.e.~$f\in H\tilde{a}_{M,m}(G)=Ha_{M,m}(G)$ and therefore we 
established the claim. But now $f\in HV_m(G)\subseteq \mathcal{V}H(G)$ holds and we are done.
\end{proof}
\noindent{}The proof of the next proposition is very similar to that of \ref{O-algebraic-equality}; therefore we omit the proof and refer to \cite{PHD} for details.  We need (as it is needed for the corresponding results on continuous functions, cf.~\cite[3.10]{ABB2009}) that the steps $(\mathcal{A}_N)_0H(G)$ of the PLB-space $(AH)_0(G)$ are complete. Unfortunately, there is (in contrast to the continuous case, cf.~\cite[Section 2]{ABB2009}) no characterization of completeness of the LB-space $\mathcal{V}_0H(G)$ for a decreasing sequence $\mathcal{V}$ of weights. However, under the assumption that $\mathcal{V}_0H(G)\subseteq \mathcal{V}_0C(G)$ is a topological subspace, $\mathcal{V}_0H(G)$ is complete if $\mathcal{V}$ is regularly decreasing, cf.~\cite[Corollary C]{Bierstedt2001}. In the setting of class $\mathcal{W}$ the latter is satisfied and hence in this case it is possible to replace the completeness assumption in \ref{o-algebraic-equality} (and also in \ref{Necessary conditions for equality}.(2)) by 
the (a priori stronger but more accessible) requirement that $\mathcal{A}_N=(a_{N,n})_{n\in\mathbb{N}}$ is regularly decreasing for each $N$.

\begin{prop}\label{o-algebraic-equality} Assume that we are in the balanced setting. If $\mathcal{A}$ satisfies condition $\text{(B)}^{\sim}$ then $(AH)_0(G)=\mathcal{V}_0H(G)$ holds algebraically. If all $(\mathcal{A}_N)_0H(G)$ are complete, the converse is also true.\hfill\qed\end{prop}
\noindent{}In the rest of this section we study the topological equalities $AH(G)=\mathcal{V}H(G)$ and $(AH)_0(G)=\mathcal{V}_0H(G)$ by using our results of Sections \ref{Necessary} and \ref{Sufficient}.

\begin{thm}\label{Necessary conditions for equality} Assume that we are in the balanced setting.\vspace{-2pt}
\begin{itemize} \item[(1)] If $AH(G)\hspace{-2.6pt}=\mathcal{V}H(G)$ holds algebraically and topologically then $\mathcal{A}$ satisfies the conditions $\text{(B)}^{\sim}$ and $\text{(wQ)}^{\sim}_{\text{in}}$.\vspace{-2pt}
\item[(2)] Assume that all $(\mathcal{A}_N)_0H(G)$ are complete. If  $(AH)_0(G)=\mathcal{V}_0H(G)$ holds algebraically and topologically then $\mathcal{A}$ satisfies the conditions $\text{(B)}^{\sim}$ and $\text{(wQ)}^{\sim}_{\text{in}}$.
\end{itemize}
\end{thm}
\begin{proof}We show (1); the proof of (2) is very similar. Condition $\text{(B)}^{\sim}$ follows with \ref{O-algebraic-equality} from the algebraical equality. The topological equality implies that $AH(G)$ is ultrabornological as it is isomorphic to an LF-space. With \ref{strong necessary condition} it follows that $\mathcal{A}$ satisfies condition $\text{(wQ)}^{\sim}_{\text{in}}$.
\end{proof}

\begin{thm}\label{O-topological-equality}Let $\mathcal{A}\subseteq W$. If $\mathcal{A}$ satisfies the conditions $\text{(B)}^{\sim}$ and $\text{(Q)}^{\sim}_{\text{out}}$ then $AH(\mathbb{D})=\mathcal{V}H(\mathbb{D})$ holds algebraically and topologically.
\end{thm}
\begin{proof}By \ref{O-algebraic-equality} the identity $\mathcal{V}H(\mathbb{D})\rightarrow AH(\mathbb{D})$ is bijective and continuous. Since $AH(\mathbb{D})$ is ultrabornological by \ref{Q-out-sim implies Projeins Null} and $\mathcal{V}H(\mathbb{D})$ is webbed, we can apply the open mapping theorem (see e.g.~\cite[24.30]{MeiseVogtEnglisch}) and obtain that the identity is an isomorphism.
\end{proof}
\noindent{}As in earlier results, also in \ref{Necessary conditions for equality} it is valid to replace $\text{(B)}^{\sim}$ with (B) and $\text{(Q)}^{\sim}_{\text{out}}$ with (Q). To conclude this section, let us point out why we have no analog of the above result in the case of o-growth conditions: In fact, the proof of \ref{Necessary conditions for equality} relies on the ultrabornologicity of $AH(\mathbb{D})$ which is needed in order to apply the open mapping theorem and in the case of o-growth conditions we had not been able to find sufficient conditions for ultrabornologicity, cf.~our remarks after \ref{Q-out-sim implies Projeins Null}.


\section{\hspace{-10pt}Special assumptions on the defining double sequence}\label{Special}

In our results on the vanishing of $\Projeins$, ultrabornologicity, barrelledness and also on the interchangeability of projective and inductive limit the necessary and the sufficient weight conditions are not the same. In view of the corresponding proofs it is not immediate how to improve our results in order to get characterizations of the forementioned properties. Therefore it is desireable to identify additional general assumptions on the sequence $\mathcal{A}$ which allow such characterizations.
\smallskip
\\Concerning barrelledness of $(AH)_0(G)$ a characterization can be achieved by the assumption that all weights in $\mathcal{A}$ are essential in the sense of Taskinen \cite{Taskinen2001}, i.e.~for each weight $a\in\mathcal{A}$ there exists $C>0$ such that $(1/a)^{\sim}\leqslant 1/a \leqslant C(1/a)^{\sim}$ holds. Under the latter assumption it is easy to see that for a double sequence $\mathcal{A}$ conditions (wQ), $\text{(wQ)}^{\sim}_{\text{in}}$ and $\text{(wQ)}^{\sim}_{\text{out}}$ as well as conditions (Q), $\text{(Q)}^{\sim}_{\text{in}}$ and $\text{(Q)}^{\sim}_{\text{out}}$ are equivalent, respectively. Therefore we get from \ref{strong necessary condition} and \ref{A_0P(G) bornological implies A_0H(G) barrelled} that for $\mathcal{A}\subseteq W$, $W$ being closed under finite maxima and all weights in $\mathcal{A}$ being essential, the space $(AH)_0(\mathbb{D})$ is barrelled if and only if (wQ) holds. 
\smallskip
\\However, also in the case of essential weights the conditions (Q) and (wQ) are a priori not equivalent and thus the assumption of being essential does not provide a characterization in the case of O-growth conditions. On the other hand, a characterization is possible if we assume that $\mathcal{A}$ satisfies the so-called condition $(\Sigma)$, which was introduced by Bierstedt, Bonet \cite[Section 5]{BB1994} for weighted LF-spaces of continuous functions. $(\Sigma)$ is a generalization of condition (V) of Bierstedt, Meise, Summers \cite{BMS1982} and it is the canonical extension of a condition for sequence spaces introduced by Vogt \cite[5.17]{Vogt1992}. We say that a double sequence $\mathcal{A}=((a_{N,n})_{n\in\mathbb{N}})_{N\in\mathbb{N}}$ on $G$ satisfies condition $(\Sigma)$ if
$$
\textstyle\forall\:N\;\exists\: K\geqslant N\:\forall\:k\;\exists\:n\geqslant k\colon \frac{a_{N,n}}{a_{K,k}} \text{ vanishes at } \infty \text{ on } G.
$$
Analogously to the previous sections, we define condition $(\Sigma)^{\sim}$ by replacing the quotient in $(\Sigma)$ by $\tilde{a}_{N,n}/\tilde{a}_{K,k}$.

\begin{prop}\label{SpectraEquivalent} Let $\mathcal{A}$ satisfy condition $(\mathit{\Sigma})$ or $(\mathit{\Sigma})^{\sim}$. Then the spectra $\mathcal{A}_0H$ and $\mathcal{A}H$ are equivalent in the sense of Wengenroth \cite[3.1.6]{Wengenroth}.
\end{prop}
\begin{proof}Assume that $(\Sigma)^{\sim}$ holds; the case of $(\Sigma)$ is similar. Let $N$ be arbitrary and select $K\geqslant N$ according to $(\Sigma)^{\sim}$. We claim that $\mathcal{A}_KH(G)\subseteq(\mathcal{A}_N)_0H(G)$ holds. In order to show this, let $f\in\mathcal{A}_KH(G)$. We select $k$ such that $f\in Ha_{K,k}(G)$. Then there exists $b_k>0$ such that $a_{K,k}|f|\leqslant b_k$, i.e.~$|f|\leqslant \frac{b_k}{a_{K,k}}$ on $G$. By \cite[1.2.(iii) and (vi)]{BBT} this implies $|f|\leqslant b_k\big(\frac{1}{a_{K,k}}\big)^{\sim}=b_k\tilde{w}_{K,k}$. We select $n\geqslant k$ according to $(\Sigma)^{\sim}$ and compute $\frac{1}{\tilde{w}_{N,n}}|f|\leqslant b_k\frac{\tilde{w}_{K,k}}{\tilde{w}_{N,n}}$. Since by $(\Sigma)^{\sim}$ the right hand side vanishes at $\infty$ on $G$, this has also to be true for $\frac{1}{\tilde{w}_{N,n}}|f|$. Finally, we have $\tilde{w}_{N,n}\leqslant w_{N,n}$ by \cite[1.2.(i)]{BBT}, i.e.~$a_{N,n}|f|=\frac{1}{w_{N,n}}|f|\leqslant\frac{1}{\tilde{w}_{N,n}}|f|$ and therefore $a_{N,
n}|f|$ has also to vanish at $\infty$ on $G$, i.e.~$f\in H(a_{N,n})_0(G)$ holds and thus $f\in(\mathcal{A}_N)_0H(G)$, which establishes the claim. The construction above gives rise to a function $K\colon\mathbb{N}\rightarrow\mathbb{N}$, $N\mapsto K(N)$ and we may assume w.l.o.g.~that $K$ is increasing. By iteration we obtain a sequence of inclusions $(i_{K^{N+1}(1),K^N(1)})_{N\in\mathbb{N}}$ which make the diagram (with an inclusion at every unlabeled arrow)
\begin{diagram}[height=2em,width=1.7em]
\cdots &\rTo &(\mathcal{A}_{K^2(1)})_0H(G) & \rTo &\cdots &\rTo                                               & (\mathcal{A}_{K(1)})_0H(G) & \rTo &\cdots &\rTo                                          &(\mathcal{A}_1)_0H(G) \\
       &     &\dTo                         &      &       &\ruTo(4,2)^{\!\!\!\!\!\!\!\!\!\!\!i_{K^2(1),K(1)}} & \dTo                       &      &       &\ruTo(4,2)^{\!\!\!\!\!\!\!\!\!\!\!i_{K(1),1}} &\dTo                  \\
\cdots &\rTo &\mathcal{A}_{K^2(1)}H(G)     & \rTo &\cdots &\rTo                                               & \mathcal{A}_{K(1)}H(G)     & \rTo &\cdots &\rTo                                          &\mathcal{A}_1H(G)     \\
\end{diagram}
commutative. Now it is only a matter of notation to see that the spectra are equivalent in the sense of \cite[3.1.6]{Wengenroth}.
\end{proof}
\noindent{}The above together with \cite[3.1.7]{Wengenroth} yields that under the assumption that $\mathcal{A}$ satisfies $(\Sigma)$ or $(\Sigma)^{\sim}$, the equality $AH(G)=(AH)_0(G)$ holds algebraically and thus also topologically. Moreover, $\Projeins\mathcal{A}H=\Projeins\mathcal{A}_0H$ holds. By Bierstedt, Bonet \cite[5.2]{BB1994}, (wQ) and (Q) are equivalent, if $\mathcal{A}$ satisfies $(\Sigma)$. An inspection of their proof shows that also $\text{(Q)}^{\sim}_{\text{in}}$ and $\text{(wQ)}^{\sim}_{\text{in}}$ are equivalent if $\mathcal{A}$ satisfies $(\Sigma)^{\sim}$.

\begin{thm}\label{TopologicalResultsSigma}Let $\mathcal{A}\subseteq W$ and $W$ be a set of class $\mathcal{W}$. If the sequence $\mathcal{A}$ satisfies $(\mathit{\Sigma})^{\sim}$, then (ii)--(viii) of the following statements are equivalent. If $\mathcal{A}$ satisfies $(\mathit{\Sigma})$ and the weights in $\mathcal{A}$ are essential, then all of the following statements are equivalent.\vspace{3pt}
\\\begin{tabular}{rlrl}\vspace{1pt}
(i)  & $\mathcal{A}$ satisfies (Q).                              &(vi)  & $(AH)_{(0)}(\mathbb{D})$ is barrelled. \\\vspace{1pt}
(ii) & $\mathcal{A}$ satisfies $\text{(Q)}^{\sim}_{\text{in}}$.  &(vii) & $\mathcal{A}$ satisfies $\text{(wQ)}^{\sim}_{\text{out}}$. \\\vspace{1pt}
(iii)& $\mathcal{A}$ satisfies $\text{(Q)}^{\sim}_{\text{out}}$. &(viii)& $\mathcal{A}$ satisfies $\text{(wQ)}^{\sim}_{\text{in}}$. \\\vspace{1pt}
(iv) & $\Projeins \mathcal{A}_{(0)}H=0$.                         &(ix)  & $\mathcal{A}$ satisfies (wQ). \\
(v)  & $(AH)_{(0)}(\mathbb{D})$ is ultrabornological.            &      &  \\
\end{tabular}
\end{thm}
\noindent{}It is (in contrast to the case of (Q) and (wQ)) not clear if (B) and $\text{(B)}^{\sim}$ are equivalent under the assumption that all weights in $\mathcal{A}$ are essential. However, the latter is true if we assume that $\mathcal{A}$ is contained in some set of weights which is closed under finite maxima and consists of essential weights only. A combination of this assumption with that of class $\mathcal{W}$ is very natural (cf.~Section \ref{Examples}) and yields the following result.

\begin{thm}\label{EqualitiesSigma}Let $\mathcal{A}\subseteq W$ and $W$ be a set of class $\mathcal{W}$ which is closed under finite maxima. If the sequence $\mathcal{A}$ satisfies $(\mathit{\Sigma})^{\sim}$, then (ii)--(v) of the following statements are equivalent. If $\mathcal{A}$ satisfies $(\mathit{\Sigma})$ and the weights in $W$ are essential, then all of the following statements are equivalent.\vspace{3pt}
\\\begin{tabular}{rlrl}\vspace{1pt}
(i)  & $\mathcal{A}$ satisfies (Q) and (B).                                                  & (v)  & $\mathcal{A}$ satisfies $\text{(wQ)}^{\sim}_{\text{out}}$ and $\text{(B)}^{\sim}$.\\\vspace{1pt}
(ii) & $\mathcal{A}$ satisfies $\text{(Q)}^{\sim}_{\text{in}}$ and $\text{(B)}^{\sim}$.      & (vi) & $\mathcal{A}$ satisfies $\text{(wQ)}^{\sim}_{\text{in}}$ and $\text{(B)}^{\sim}$.\\\vspace{1pt}
(iii)& $\mathcal{A}$ satisfies $\text{(Q)}^{\sim}_{\text{out}}$ and $\text{(B)}^{\sim}$.     & (vii)& $\mathcal{A}$ satisfies (wQ) and (B).\vspace{-1.5pt}
\end{tabular}

\begin{tabular}{rl}
\hspace{1.5pt}(iv) & $(AH)_{(0)}(\mathbb{D})=\mathcal{V}H(\mathbb{D})$ holds algebraically and topologically.
\end{tabular}
\end{thm}


\section{\hspace{-10pt}Examples}\label{Examples}

The main example for some system $W$ which is of class $\mathcal{W}$ is the set
\begin{align*}
\mathcal{W}(\varepsilon_0,k_0)\!:=\!\big\{&w\colon\mathbb{D}\rightarrow\mathbb{R}\:;\:w \text{ is continuous, strictly positive, radial, non-}\\
                                   &\text{increasing on $[0,1[$, $w(r)\stackrel{\scriptscriptstyle r\nearrow1}{\rightarrow}0$ holds and the two estimates}\\
                                   &\text{(L1) }\!\inf_{k\in\mathbb{N}}{\textstyle\frac{w(r_{k+1})}{w(r_k)}}\geqslant \varepsilon_0 \text{ and (L2) }\!\varlimsup_{k\rightarrow\infty}{\textstyle\frac{w(r_{k+k_0})}{w(r_k)}}<1\!-\!\varepsilon_0 \text{ are valid}\big\}
\end{align*}
where $\varepsilon_0>0$ and $k_0\in\mathbb{N}$ are constants. The above formulation is due to Bierstedt, Bonet \cite{BB2003} who showed that $\mathcal{W}(\varepsilon_0,k_0)$ is of class $\mathcal{W}$. The conditions (L1) and (L2) constitute a uniform version of conditions introduced by Lusky \cite{Lusky1992, Lusky1995}, which also appear in the sequence space representation for weighted LB-spaces studied by Mattila, Saksman, Taskinen \cite{MattilaSaksmanTaskinen}. For the proof of \ref{polynomials barrelled-lemma} we assumed that $W$ is closed under finite maxima. This is not included in the definition of class $\mathcal{W}$ given by Bierstedt, Bonet \cite{BB2003} but for the above example the latter is easy to verify. Moreover, all weights in $\mathcal{W}(\varepsilon_0,k_0)$ are essential, see Bierstedt, Bonet \cite{BB2003} in combination with Shields, Williams \cite[2.1.(iv)]{ShieldsWilliams1982}. A detailed proof is contained in Doma\'{n}ski, Lindstr\"{o}m \cite{DomanskiLindstroem}. Let us now discuss 
examples for the sequence $\mathcal{A}$; for detailed computations we refer to \cite{PHD}.

\begin{exa}\label{MST-Example-1} Based on an example of \cite[3.8]{MattilaSaksmanTaskinen} we put $a_{N,n}(z):=N(1-|z|)^{\alpha\frac{n}{n+1}}$ for some $\alpha>0$. Then $\mathcal{A}=((a_{N,n})_{n\in\mathbb{N}})_{N\in\mathbb{N}}\subseteq\mathcal{W}(\varepsilon_0,k_0)$ holds for suitable $\varepsilon_0$ and $k_0$, $\mathcal{A}$ satisfies $(\Sigma)$, (wQ) and (B). By \ref{TopologicalResultsSigma} and \ref{EqualitiesSigma}, we have $AH(\mathbb{D})=(AH)_0(\mathbb{D})=\mathcal{V}H(\mathbb{D})$ algebraically and topologically, these spaces are ultrabornological and barrelled and $\Projeins\mathcal{A}_0H=\Projeins\mathcal{A}H=0$ holds.
\end{exa}
\noindent{}\ref{MST-Example-1} is in some sense the easiest way to construct a double sequence using the example of \cite[3.8]{MattilaSaksmanTaskinen}: We multiplied each weight $v_n(z)=(1-|z|)^{\alpha\frac{n}{n+1}}$ of the sequence $\mathcal{V}=(v_n)_{n\in\mathbb{N}}$ (which was studied by Mattila, Saksman, Taskinen in the context of LB-spaces) with $N$. In order to produce more examples it is natural to generalize \ref{MST-Example-1} in the following sense: We consider sequences $\mathcal{A}=((a_{N,n})_{n\in\mathbb{N}})_{N\in\mathbb{N}}$ with $a_{N,n}=a_N\cdot{}v_n$ with an increasing sequence $(a_{N})_{N\in\mathbb{N}}$ and a decreasing sequence $(v_n)_{n\in\mathbb{N}}$.
\smallskip
\\A first idea concerning concrete examples for $a_N$ and $v_n$ might be to set $a_N(z):=(1-|z|)^{\varepsilon_N}$ and  $v_n(z):=(1-|z|)^{\delta_n}$, i.e.~$a_{N,n}(z)=(1-|z|)^{\varepsilon_N+\delta_n}$ for sequences $\varepsilon_N\searrow\varepsilon\in[0,\infty[$ and $\delta_n\nearrow\delta\in\;]0,\infty]$. Unfortunately, it turns out that for this selection of $\mathcal{A}$, $AH(\mathbb{D})$ never is a proper PLB-space, more precisely: $AH(\mathbb{D})$ is an LB-space if $\delta=\infty$ and it is a Fr\'{e}chet space otherwise (see \ref{Bonet-Proper}.(ii) and the second conditions in the proof of \ref{Bonet-Proper}.(ii)). Let us thus define $a_N(z):=a(|z|)^{\varepsilon_N}$ and $v_n(z):=v(|z|)^{\delta_n}$ for $(\varepsilon_N)_{N\in\mathbb{N}}$ and $(\delta_n)_{n\in\mathbb{N}}$ as above and (distinct) decreasing functions $a$, $v\colon [0,1[\:\rightarrow\:]0,1]$. The following result (which was communicated to the author by Bonet) establishes a sufficient condition on $a$ and $v$ assuring that $AH(\mathbb{D})$ is 
a proper PLB-space.

\begin{prop}\label{Bonet-Proper}(Bonet, 2010) Let $a$, $v\colon [0,1[\:\rightarrow\:]0,1]$ be continuous, decreasing with $\lim_{r\nearrow1}v(r)=\lim_{r\nearrow1}a(r)=0$. Let $(\varepsilon_N)_{N\in\mathbb{N}}$ and $(\delta_n)_{n\in\mathbb{N}}$ be sequences with $\varepsilon_N\searrow\varepsilon\in[0,\infty[$ and $\delta_n\nearrow\delta\in\;]0,\infty]$ such that $a^{\varepsilon_N}$ and $v^{\delta_n}$ are essential for all $N$ and $n$. Let $\mathcal{A}=((a_{N,n})_{n\in\mathbb{N}})_{N\in\mathbb{N}}$ be defined by $a_{N,n}(z)=a(|z|)^{\varepsilon_N}v(|z|)^{\delta_n}$ for $z\in\mathbb{D}$.
\begin{itemize}
\item[(i)] $AH(\mathbb{D})$ is a DF-space if and only if $\mathcal{A}$ satisfies condition (df), i.e.~there exists $P$ such that for all $N\geqslant P$ and all $n$ there exists $m>n$ and $C>0$ such that $a_{N,m}\leqslant Ca_{P,n}$ holds.
\item[(ii)] $AH(\mathbb{D})$ is metrizable if and only if $\mathcal{A}$ satisfies condition (m), i.e.~for every $N$ there exists $M>N$ and $n$ such that for all $m\geqslant n$ there exists $C>0$ such that $a_{N,n}\leqslant a_{M,m}$ holds.
\item[(iii)] If $\limsup_{r\nearrow1}\log a(r)/\log v(r)=\limsup_{r\nearrow1}\log v(r)/\log a(r)=\infty$\linebreak{}then $AH(\mathbb{D})$ is neither a DF-space nor matrizable and thus in particular not an LB- and also not a Fr\'{e}chet space.
\end{itemize}
\end{prop}
\begin{proof} Let us first note (as a preparation for the proofs of (i) and (ii)), that in our setting the sequence $\mathcal{A}_N:=(a_{N,n})_{n\in\mathbb{N}}$ satisfies condition (S) (or (V)), i.e.~for any $n$ there exists $m>n$ such that $a_{N,m}/a_{N,n}$ vanishes at infinity of $\mathbb{D}$, of Bierstedt, Meise, Summers \cite[p.~108]{BMS1982} for any $N$. \medskip
\\(i) \textquotedblleft{}$\Leftarrow$\textquotedblright{} We select $P$ as in (df) and claim that $\mathcal{A}H(\mathbb{D})=\mathcal{A}_PH(\mathbb{D})$ holds. Clearly, $AH(\mathbb{D})\subseteq\mathcal{A}_P(\mathbb{D})$ holds with continuous inclusion. In view of the open mapping theorem (e.g.~\cite[24.30]{MeiseVogtEnglisch}) it is enough to show that $\mathcal{A}_P(\mathbb{D})\subseteq AH(\mathbb{D})$ holds algebraically. Let $f\in\mathcal{A}_PH(\mathbb{D})$, i.e.~there exists $n$ such that $f\in Ha_{P,n}(\mathbb{D})$. For $N\geqslant{}P$ and $n$ as before we have ~$\sup_{z\in\mathbb{D}}a_{N,m}(z)|f(z)|\leqslant{}C\sup_{z\in\mathbb{D}}a_{P,n}(z)|f(z)|<\infty$ and thus $f\in Ha_{N,m}(\mathbb{D})\subseteq\mathcal{A}_NH(\mathbb{D})$ holds. For $N<P$ we have $\sup_{z\in\mathbb{D}}a_{N,n}(z)|f(z)|\leqslant\sup_{z\in\mathbb{D}}a_{P,n}(z)|f(z)|<\infty$ and thus $f\in Ha_{N,n}(\mathbb{D})\subseteq\mathcal{A}_NH(\mathbb{D})$. Since we thus have $\mathcal{A}_PH(\mathbb{D})\subseteq\mathcal{A}_NH(\mathbb{D})$ for all 
$N$ the conclusion follows.
\smallskip
\\\textquotedblleft{}$\Rightarrow$\textquotedblright{} Condition (S) implies that $\mathcal{A}_NH(\mathbb{D})=(\mathcal{A}_N)_0H(\mathbb{D})$ for any $N$, cf.~\cite[0.4]{BMS1982}. Moreover, each $\mathcal{A}_N$ satisfying (S) implies that $\mathcal{A}$ satisfies $\text{(}\Sigma\text{)}$. Since our assumptions guarantee that we are in the balanced setting (see Section \ref{Necessary}) we get from the above (and \ref{SpectraEquivalent}) that the inclusions $AH(\mathbb{D})\subseteq\mathcal{A}_NH(\mathbb{D})\subseteq\mathcal{A}_MH(\mathbb{D})$ both hold with dense image whenever $N\geqslant M$ is satisfied. Thus, in the inductive spectrum $(\mathcal{A}_NH(\mathbb{D})'_b)_{N\in\mathbb{N}}$ of Fr\'{e}chet spaces the natural linking maps are injective. Moreover, each step may be regarded as a subspace of $AH(\mathbb{D})'_b$ with continuous inclusion map. Since by the above $AH(\mathbb{D})$ is a reduced projective limit, we get from K\"othe \cite[statement (6) on p.~290]{KoetheI} that $AH(\mathbb{D})'_b=\ind{N} \mathcal{A}_NH(\mathbb{D})'_b$ holds algebraically. By our assumptions, $AH(\mathbb{D})$ is a DF-space and thus its strong dual is a Fr\'{e}chet space (e.g.~\cite[25.9]{MeiseVogtEnglisch}).
\smallskip
\\The facts collected so far can be formulated as follows: Consider the identity map $AH(\mathbb{D})'_b\rightarrow\Bigcup{N\in\mathbb{N}}\mathcal{A}_NH(\mathbb{D})'_b$ where the space on the right hand side is endowed with the topology of $AH(\mathbb{D})'_b$. Then this map is continuous, every inclusion of $\mathcal{A}_NH(\mathbb{D})'_b$ into the space on the right hand side is continuous and the space on the left hand side is a Fr\'{e}chet space. Thus, we may apply Grothendieck's factorization theorem (e.g.~\cite[22.33]{MeiseVogtEnglisch}) to conclude that there exists $P$ such that $AH(\mathbb{D})'_b\subseteq\mathcal{A}_PH(\mathbb{D})'_b$ holds with continuous inclusion. Thus, for any $N\geqslant P$ we obtain $\mathcal{A}_PH(\mathbb{D})'_b=\mathcal{A}_NH(\mathbb{D})'_b$ algebraically and topologically.
\smallskip
\\Let us now show that the above implies already that $\mathcal{A}_PH(\mathbb{D})=\mathcal{A}_NH(\mathbb{D})$ holds algebraically and topologically. The latter spaces are both complete (see Section \ref{Preliminaries}) and we know already that $\mathcal{A}_NH(\mathbb{D})$ is dense in $\mathcal{A}_PH(\mathbb{D})$. Therefore it is enough to show that $\mathcal{A}_NH(\mathbb{D})\subseteq\mathcal{A}_PH(\mathbb{D})$ is a topological subspace. The topology of $\mathcal{A}_NH(\mathbb{D})$ is given by the system of seminorms $(p_M)_{M\in\mathcal{M}}$ with $p_M(x)=\sup_{y\in M}|y(x)|$ for $x\in E$ and $\mathcal{M}$ being the system of equicontinuous subsets of $\mathcal{A}_NH(\mathbb{D})'_b$. Since $\mathcal{A}_NH(\mathbb{D})$ is barrelled, we get the same topology if we replace $\mathcal{M}$ by the system of bounded subsets of $\mathcal{A}_NH(\mathbb{D})'_b$. Since the same arguments apply to $\mathcal{A}_PH(\mathbb{D})$, the coincidence of the strong duals yields that the seminorms which generate the topology of $\mathcal{A}_NH(\mathbb{D})$ are just the restrictions of those generating the topology of $\mathcal{A}_PH(\mathbb{D})$, i.e.~$\mathcal{A}_NH(\mathbb{D})$ is a topological subspace of $\mathcal{A}_PH(\mathbb{D})$.
\smallskip
\\In order to show that (df) holds we select $P$ as above. For $N\geqslant P$ and arbitrary $n$ we then have $Ha_{P,n}(\mathbb{D})\subseteq\mathcal{A}_NH(\mathbb{D})$ with continuous inclusion. Therefore, we may apply Grothendieck's factorization theorem a second time to conclude that there exists $m$ such that $Ha_{P,n}(\mathbb{D})\subseteq Ha_{N,m}(\mathbb{D})$ holds with continuous inclusion. Thus, there exists $C>0$ such that $B_{P,n}\subseteq C B_{N,m}$ holds. By the definition of associated weights, this inclusion yields $\tilde{a}_{N,m}\leqslant C\tilde{a}_{P,n}$. Due to our assumptions concerning essentialness and in view of \cite[Remarks previous to 1.5]{BBT} it follows $a_{N,m}\leqslant C a_{P,n}$.
\medskip
\\(ii) \textquotedblleft{}$\Leftarrow$\textquotedblright{} We show that there exist increasing sequences $(N(j))_{j\in\mathbb{N}}$ and $(n(j))_{j\in\mathbb{N}}$ such that $AH(\mathbb{D})=F$ holds where $F=\Bigcap{j\in\mathbb{N}}Ha_{N(j),n(j)}(\mathbb{D})$ is endowed with the topology of the seminorms $(p_L)_{L\in\mathbb{N}}$ where $p_L(f)=\max_{j=1,\dots,L}\sup_{z\in\mathbb{D}}a_{N(j),n(j)}(z)|f(z)|$.
\smallskip
\\We define the forementioned sequences iteratively: We put $N(1):=1$. For $N=1$ we select $M$ and $n$ as in (m) and define $N(2):=N$ and $n(1):=n$. By (m) we have $Ha_{N(2),m}(\mathbb{D})\subseteq Ha_{N(1),n(1)}(\mathbb{D})$ with continuous inclusion for all $m$. Thus, $\mathcal{A}_{N(2)}H(\mathbb{D})\subseteq Ha_{N(1),n(1)}(\mathbb{D})$ holds with continuous inclusion. Given $N(2)$ as above we select $N(3)>N(2)$ arbitrary and (for $N=N(3)$) we select $n(2):=n$ according to (m). Then again $Ha_{N(3),m}(\mathbb{D})\subseteq Ha_{N(2),n(2)}(\mathbb{D})$ and therefore $\mathcal{A}_{N(3)}H(\mathbb{D})\subseteq Ha_{N(2),n(2)}(\mathbb{D})$ holds with continuous inclusion. Proceeding in this way we obtain $(N(j))_{j\in\mathbb{N}}$, $(n(j))_{j\in\mathbb{N}}$ such that $\mathcal{A}_{N(j+1)}H(\mathbb{D})\subseteq Ha_{N(j),n(j)}(\mathbb{D})$ holds with continuous inclusion for any $j$.
\smallskip
\\$F\subseteq AH(\mathbb{D})$ holds with continuous inclusion and from the last paragraph we get that also $AH(\mathbb{D})\subseteq\proj{j}\mathcal{A}_{N(j+1)}H(\mathbb{D})\subseteq F$ with continuous inclusion.
\smallskip
\\\textquotedblleft{}$\Rightarrow$\textquotedblright{} By (S) and \cite[1.6]{BMS1982} projective description (see the survey \cite{Bierstedt2001} of Bier\-stedt for detailed information on this notion) holds for every step of $AH(\mathbb{D})$, i.e.~the topology of $\mathcal{A}_NH(\mathbb{D})$ is given by the system of seminorms $(p_{\bar{a}})_{\bar{a}\in\bar{A}_N}$ with $p_{\bar{a}}(f)=\sup_{z\in\mathbb{D}}\bar{a}(z)|f(z)|$ for $f\in\mathcal{A}_NH(\mathbb{D})$ and
$$
\bar{A}_N=\{\bar{a}\colon \mathbb{D}\rightarrow\mathbb{R}\:;\:\exists\:\bar{v}\in\bar{V}\colon \bar{a}=a^{\varepsilon_N}\bar{v}\},
$$
$$
\bar{V}=\{\bar{v}\colon\mathbb{D}\rightarrow\:]0,1[\;;\:\forall\:n\;\exists\:\lambda_n>0\colon\bar{v}\leqslant\lambda_nv^{\delta_n}\text{ on }\mathbb{D}\}
$$
where we may assume that all members of $\bar{V}$ are radial and continuous, see \cite{BMS1982}. Therefore, a fundamental system of seminorms of $AH(\mathbb{D})$ is given by $(p_{w})_{w\in W}$ with $p_w(f)=\sup_{z\in\mathbb{D}}w(z)|f(z)|$ for $f\in AH(\mathbb{D})$ and
$$
W=\{w\colon\mathbb{D}\rightarrow\mathbb{R}\:;\:\exists\:N,\,\bar{v}\in\bar{V}\colon w=a^{\varepsilon_N}\bar{v}\}.
$$
If $AH(\mathbb{D})$ is metrizable we can select an increasing sequence $(w_k)_{k\in\mathbb{N}}$, $w_k=a^{\varepsilon_{M(k)}}\bar{v}_k$, $M(k+1)>M(k)$, $\bar{v}_k\in\bar{V}$ such that for every continuous seminorm $p$ on $AH(\mathbb{D})$ there exists $k=k(p)$ such that $p(f)\leqslant\sup_{z\in\mathbb{D}}w_k(z)|f(z)|$ holds for all $f\in AH(\mathbb{D})$. We consider now the sequence $(k(p))_{p\in\text{cs}(AH(\mathbb{D}))}$. For every $k$ and for every $n$ there is $\lambda_{n}^{(k)}>0$ such that $\bar{v}_k\leqslant\lambda_n^{(k)}v^{\delta_n}$ and thus $w_k=a^{\varepsilon_{M(k)}}\bar{v}_k\leqslant\lambda_n^{(k)}a^{\varepsilon_{M(k)}}v^{\delta_n}=\lambda_n^{(k)}a_{M(k),n}$ holds on $\mathbb{D}$. Consequently, $Ha_{M(k),n}(\mathbb{D})\subseteq Hw_k(\mathbb{D})$ holds with continuous inclusion for every $n$ and thus $\mathcal{A}_{M(k)}H(\mathbb{D})\subseteq Hw_k(\mathbb{D})$ holds with continuous inclusion. To sum up so far, we have
$$
(\star)\;\;\;\;\;\forall\:k\;\exists\:M\colon\mathcal{A}_{M}H(\mathbb{D})\subseteq Hw_k(\mathbb{D}) \text{ with continuous inclusion.}
$$
Next, we claim that for any $N$ there exist $k=k(N)$, $n=n(N)$ and $C_N>0$ such that $a_{N,n(N)}|g|\leqslant C_N$ holds for all $g\in\mathbb{P}$ with $w_k|g|\leqslant1$. Assume that this is not the case. Then there exists $N_0$ such that for all $k=n$ and $C_N=k^2$ there is $g_k\in\mathbb{P}$ with $w_k|g_k|\leqslant1$ on $\mathbb{D}$ and $a_{N,k}|g_k|>k^2$ at some point on $\mathbb{D}$. Now $(g_k/k)_{k\in\mathbb{N}}\subseteq AH(\mathbb{D})$ and $w_k|g_k/k|\leqslant1/k$ and thus $g_k/k\rightarrow0$ in $AH(\mathbb{D})$. Therefore, $g_k/k\rightarrow0$ in $\mathcal{A}_NH(\mathbb{D})$ for every $N$ and there is $m$ such that $(g_k/k)_{k\in\mathbb{N}}\subseteq Ha_{N,m}(\mathbb{D})$ is bounded. Hence, there exist $M$ and $D>0$ such that $a_{N,m}|g_k|\leqslant Dk$ on $\mathbb{D}$ for every $k$, a contradiction. Let us now show that for given $N$ and $k=k(N)$, $n=n(N)$ as above $Hw_k(\mathbb{D})\subseteq Ha_{N,n}(\mathbb{D})$ holds with continuous inclusion. Given $g\in Hw_k(\mathbb{D})$ satisfying $w_k|g|\leqslant1$ 
on $\mathbb{D}$ we consider $(S_jg)_{j\in\mathbb{N}}\subseteq\mathbb{P}$ satisfying $w_k|S_jg|\leqslant1$ and $S_jg\rightarrow g$ in $H(\mathbb{D})$. By the above, $a_{N,n}|g_j|\leqslant C_N$ holds for all $j$. Taking the pointwise limit for $j$ tending to infinity we get $a_{N,n(N)}|g|\leqslant C_N$ from where the continuous inclusion follows. We thus have shown
$$
(\circ)\;\;\;\;\;\forall\:N\;\exists\:k,\,n\colon Hw_k(\mathbb{D})\subseteq Ha_{N,n}(\mathbb{D}) \text{ with continuous inclusion.}
$$
To show (m) let now $N$ be given. We select $k$ as in $(\circ)$ and for this $k$ we select $M$ as in $(\star)$. By increasing $k$ we can assume $M>N$. Moreover, we select $n$ as in $(\circ)$. Then we have $\mathcal{A}_{M}H(\mathbb{D})\subseteq Ha_{N,n}(\mathbb{D})$ with continuous inclusion. Thus, for $m\geqslant n$ it follows $Ha_{M,m}(\mathbb{D})\subseteq Ha_{N,n}(\mathbb{D})$  with continuous inclusion. Arguments similar to those at the end of the proof of (i) show the estimate in (m).
\medskip
\\(iii) Straightforward computations show that condition (df) is equivalent to
$$
\exists\:P\:\forall\:N\geqslant P,\,n\;\exists\:m>n,\,r_0\in\;]0,1[\;\forall\:r\geqslant r_0\colon {\textstyle\frac{\delta_m-\delta_n}{\varepsilon_P-\varepsilon_N}\geqslant\frac{\log a(r)}{\log v(r)}}.
$$
and that condition (m) is equivalent to
$$
\forall\:N\;\exists\:M>N,\,n\;\forall\:m>n\;\exists\:r_0\in\;]0,1[\;\forall\:r\geqslant r_0\colon {\textstyle\frac{\varepsilon_N-\varepsilon_M}{\delta_m-\delta_n}\geqslant\frac{\log v(r)}{\log a(r)}}.
$$
Now it is obvious that none of these two conditions can be satisfied under the assumption stated in (iii).
\end{proof}

\noindent{}Since the condition in \ref{Bonet-Proper}.(iii) at first glance looks hard to actualize, let us give a hint at a possible construction of $a$ and $v$ satisfying this condition. Proceeding as Bierstedt, Bonet \cite[Claim on p.~765]{BB2006} for given $\alpha$, $\beta\colon [1,\infty[\:\rightarrow\:[1,\infty[$ (Bierstedt, Bonet suggest to think of $\alpha(t)=t^2$ and $\beta(t)=t^3$) we construct $\gamma\colon[1,\infty[\:\rightarrow\:[1,\infty[$ continuous, strictly increasing, convex with $\alpha\leqslant\gamma\leqslant\beta$ and sequences $(a_k)_{k\in\mathbb{N}}$ and $(b_k)_{k\in\mathbb{N}}$ tending to infinity with $1=a_1<b_1<a_2<\cdots<a_k<b_k<\cdots$ such that $\gamma(a_k)=\beta(a_k)$ and $\gamma(b_k)=\alpha(b_k)$ holds for all $k$. Now we may select $\nu\colon[1,\infty[\:\rightarrow\:[1,\infty[$ with $\alpha<\nu<\beta$ such that $\beta(t)/\nu(t)$ and $\nu(t)/\alpha(t)$ tend to infinity for $t$ tending to infinity (in the above example we may for instance take $\nu(t)=t^{5/2}$). By the properties 
of $\gamma$ we thus get $\limsup_{t\rightarrow\infty}\gamma(t)/\nu(t)=\limsup_{t\rightarrow\infty}\nu(t)/\gamma(t)=\infty$. Finally we put $a(r)=(\exp(1/\gamma(1/(1-r)))-1)/\exp(1)$ and $v(r)=(\exp(1/\nu(1/(1-r)))-1)/\exp(1)$ which then by construction satisfy the condition in \ref{Bonet-Proper}.(iii). With the help of Bonet, Doma\'{n}ski, Lindstr\"om \cite[Proposition 7]{BDL} it can (for reasonable $\nu$) then be concluded that $a^{\varepsilon_N}$ and $v^{\delta_n}$ are essential and the space $AH(\mathbb{D})$ corresponding to the sequence $\mathcal{A}=((a_{N,n})_{n\in\mathbb{N}})_{N\in\mathbb{N}}$ with $a_{N,n}(z)=a(|z|)^{\varepsilon_N}v(|z|)^{\delta_n}$ will indeed be a proper PLB-space.

\subsection*{Acknowledgements}

This article arises from a part of the author's doctoral thesis. He is indebted to Klaus D.~Bierstedt who guided his interest to weighted function spaces and gave him the opportunity for this research under his direction. After the sudden death of Klaus D.~Bierstedt, the author continued his work under the direction of Jos\'{e} Bonet; for this and also for many advices, helpful remarks and valuable suggestions the author likes to thank him sincerely.
\smallskip
\\In addition, the author likes to thank the referee who made him aware of the fact that the first idea on further examples (explained after \ref{MST-Example-1}) does not yield proper PLB-spaces and to Jos\'{e} Bonet who communicated the result \ref{Bonet-Proper}, its consequences and the idea on the construction explained at the end of Section \ref{Examples} to the author.

\setlength{\parskip}{0cm}

\renewcommand{\baselinestretch}{1}\normalsize

\noindent{\sc Author's Address:}{ \footnotesize Sven-Ake Wegner, Fachbereich C -- Mathematik und Naturwissenschaften, Arbeitsgruppe Funktionalanalysis, Bergische Universit\"at Wuppertal, Gau\ss{}\-stra\ss{}e 20, D-42097 Wuppertal, GERMANY, e-mail: wegner@math.uni-wuppertal.de.}

\end{document}